\documentclass{amsart}
\usepackage[english]{babel}
\usepackage[latin1]{inputenc}
\usepackage[dvips,final]{graphics}
\usepackage{amsmath,amsfonts,amssymb,amsthm,amscd,array,
stmaryrd,mathrsfs,mathdots,epigraph}
\usepackage[makeroom]{cancel}
\usepackage{pstricks}
\usepackage[all]{xy}
\usepackage{url}
\usepackage{multirow, blkarray}
\usepackage{booktabs}
\usepackage{caption}
\usepackage{subcaption}

\usepackage{blindtext}
\usepackage{textcomp}
\usepackage[final]{epsfig}
\usepackage{color}
\usepackage{mathabx}

\theoremstyle{plain}
\newtheorem{thm}{Theorem}[section]
\newtheorem{lem}[thm]{Lemma}
\newtheorem{prop}[thm]{Proposition}
\newtheorem{cor}[thm]{Corollary}

\theoremstyle{definition}
\newtheorem*{defn}{Definition}
\newtheorem*{rem}{Remark}
\newtheorem*{rems}{Remarks}
\newtheorem*{ex}{Example}
\newtheorem*{exs}{Examples}
\newtheorem*{pb}{Problem}

\numberwithin{figure}{section}

\let\ssection=\section
\renewcommand{\section}{\setcounter{equation}{0}\ssection}

\newcommand{\R}{\mathbb{R}}
\newcommand{\Z}{\mathbb{Z}}

\newcommand{\N}{\mathbb{N}}
\newcommand{\bF}{\mathbb{F}}

\newcommand{\bP}{\mathbb{P}}
\newcommand{\Q}{\mathbb{Q}}

\newcommand{\cC}{\mathcal{C}}
\newcommand{\cD}{\mathcal{D}}

\newcommand{\Id}{\mathrm{Id}}
\newcommand{\SL}{\mathrm{SL}}
\newcommand{\PSL}{\mathrm{PSL}}

\def\thup{\mathop{\rm th}\nolimits}
\def\vecm{\bar{m}}
\def\vecv{\bar{v}}
\def\vecw{\bar{w}}

\begin{document}

\title[Quiddities of dissections]
{Quiddities of polygon dissections\\
and the Conway-Coxeter frieze equation}

\author{Charles H.\ Conley}
\address{
Charles H.\ Conley,
Department of Mathematics 
\\University of North Texas 
\\Denton TX 76203, USA} 
\email{conley@unt.edu}

\author{Valentin Ovsienko}
\address{
Valentin Ovsienko,
Centre National de la Recherche Scientifique,
Laboratoire de Math\'ematiques de Reims, UMR9008 CNRS,
Universit\'e de Reims Champagne-Ardenne,
U.F.R. Sciences Exactes et Naturelles,
Moulin de la Housse - BP 1039,
51687 Reims cedex 2,
France}
\email{valentin.ovsienko@univ-reims.fr}



\begin{abstract}
We study a $2\times2$ matrix equation arising naturally
in the theory of Coxeter frieze patterns.
It is formulated in terms of the generators
of the group $\mathrm{PSL}(2,\mathbb{Z})$ and
is closely related to continued fractions.
It appears in a number of different areas,
for example, toric varieties.
We count its positive solutions,
obtaining a series of integer sequences,
some known and some new.
This extends classical work of Conway and Coxeter proving
that the first of these sequences is the Catalan numbers.
\end{abstract}

\thanks{\noindent
C.H.C.\ was partially supported by a Simons Foundation Collaboration Grant, 519533.\\
\indent V.O.\ was partially supported by the ANR project PhyMath, ANR-19-CE40-0021.
}

\maketitle

\hfill \textit{To the memory of John Conway}

\thispagestyle{empty}

\tableofcontents

\section{Introduction}

Consider the $2\times2$ matrix equation
\begin{equation}
\label{SLMatEqInt}
\begin{pmatrix}
a_1 & -1 \\ 1 & \phantom{-}0
\end{pmatrix}
\begin{pmatrix}
a_2 & -1 \\ 1 & \phantom{-}0
\end{pmatrix}
\cdots
\begin{pmatrix}
a_N & -1 \\ 1 & \phantom{-}0
\end{pmatrix}
= \pm \Id,
\end{equation}
where the indeterminates $(a_1,a_2,\ldots,a_N)$ are integers.

\begin{defn}
\begin{itemize}
\item
We shall refer to \eqref{SLMatEqInt} as the \textit{Conway-Coxeter equation.}
\smallbreak \item
We shall refer to a solution of \eqref{SLMatEqInt} as \textit{positive} if all of the integers $a_i$ are positive.
\end{itemize}
\end{defn}

The Conway-Coxeter equation arises in the theory of Coxeter frieze patterns~\cite{Cox}
and has been studied in several articles, such as \cite{BR, CoOv2, Sop, SVRS, Val}.
It is also relevant for several other fields in algebra, geometry, and combinatorics,
including the theory of two-dimensional toric varieties; see~\cite{Ful}, Section~2.5.
The project of enumerating all positive solutions
of~\eqref{SLMatEqInt} was begun in~\cite{Val}.
The main purpose of this article is to complete it.

Let us write $T$ for the sum of the integers $a_i$:
\begin{equation} \label{MatEq Total Sum}
T(a_1, \ldots, a_N) := a_1+a_2+\cdots+a_N.
\end{equation}
This quantity is an important characteristic of the collection of
matrix factors in~\eqref{SLMatEqInt}.  As discussed in~\cite{Val},
it has several combinatorial and dynamical interpretations.
We will refer to it as the \textit{total sum.}
In addition to enumerating all positive solutions of~\eqref{SLMatEqInt},
we also enumerate those with any given fixed value of $T$.

It turns out that for any positive solution of~\eqref{SLMatEqInt}, 
the total sum necessarily satisfies
\begin{equation}
\label{AnyA}
T = 3(N - 2) - 6k
\end{equation}
for some non-negative integer $k \leq \frac{1}{3}N - 1$.
It can be shown that the right side
of~\eqref{SLMatEqInt} is $(-1)^{k+1} \Id$.
The Conway-Coxeter solutions arising from Coxeter's
frieze patterns correspond to $k=0$ (see Section~\ref{CoCoS}).
The case $k=1$ is also of particular significance (see Section~\ref{ProjPlane}).

The Conway-Coxeter equation has a certain ubiquity.
Various combinatorial and geometric problems can be formulated
in terms of it, with differing conditions on the indeterminates $a_i$.
The positivity condition leads to interesting combinatorics,
bearing out the general principle that any
naturally occurring sequence of positive integers
must enumerate some concrete set of objects.

\subsection{Hirzebruch-Jung continued fractions}
There is a close relation between \eqref{SLMatEqInt} and
the Hirzebruch-Jung continued fraction \cite{Ful, Hir},
\begin{equation*}
\llbracket{}a_1,\ldots,a_N\rrbracket{}\;:=\;
a_1 - \cfrac{1}{a_2 - \cfrac{1}{\ddots - \cfrac{1}{a_N} } } \ \ ,
\end{equation*}
sometimes called in the literature the ``negative'', ``minus'',
or ``reversal'' continued fraction~\cite{Kat}.
It is classical that the matrix product in \eqref{SLMatEqInt}
encodes this continued fraction; see for example~\cite{MGO}.
To be precise, consider the tridiagonal determinant
\begin{equation*}
K_N(a_1,\ldots,a_N):=
\det\left(
\begin{array}{cccccc}
a_1&1&&&\\[4pt]
1&a_{2}&1&&\\[4pt]
&\ddots&\ddots&\!\!\ddots&\\[4pt]
&&1&a_{N-1}&\!\!\!\!\!1\\[6pt]
&&&\!\!\!\!\!1&\!\!\!\!a_N
\end{array}
\right),
\end{equation*}
known as Euler's \textit{continuant.}
The Hirzebruch-Jung continued fraction
is a quotient of two continuants,
$$
\llbracket{} a_1, \ldots, a_N \rrbracket =
\frac{K_N(a_1,\ldots,a_N)}{K_{N-1}(a_2,\ldots,a_N)},
$$
while the matrix product in \eqref{SLMatEqInt} is
\begin{equation*}
\begin{pmatrix}
K_N(a_1,\ldots,a_N)&-K_{N-1}(a_1,\ldots,a_{N-1})\\[6pt]
K_{N-1}(a_2,\ldots,a_N)&-K_{N-2}(a_2,\ldots,a_{N-1})
\end{pmatrix}.
\end{equation*}

In fact, \eqref{SLMatEqInt} is equivalent to the system of integer equations
\begin{eqnarray*}
K_{N-1}(a_2,\ldots,a_N) = 0, \qquad
K_{N-1}(a_1,\ldots,a_{N-1}) = 0.
\end{eqnarray*}
These conditions imply
$K_N(a_1,\ldots,a_N) = K_{N-2}(a_2,\ldots,a_{N-1}) = \pm1$,
as the determinant of \eqref{SLMatEqInt} is necessarily~$1$.
In this situation many authors speak of
$\llbracket{}a_1,\ldots,a_{N-1}\rrbracket$
and $\llbracket{}a_2,\ldots,a_N\rrbracket$
as ``continued fractions representing zero''; see \cite{Chr, HTU, Ste}.
From this point of view, the question we answer in this article
may be formulated as follows:
in how many ways can zero be represented
by a Hirzebruch-Jung continued fraction
$\llbracket{} a_1, \ldots, a_N \rrbracket$
such that $a_1,\ldots,a_N$ are positive integers?

\subsection{Triangulations and the Conway-Coxeter theorem} \label{CoCoS}
A theorem of Conway and Coxeter identifies
a class of solutions of \eqref{SLMatEqInt} which correspond
to triangulations of convex $N$-gons by non-crossing diagonals.
In order to state it they introduce the notion of \textit{quiddity:}
the quiddity of a triangulation is the cyclically ordered $N$-tuple
$(a_1,\ldots,a_N)$, where $a_i$ is the number of triangles
contacting the $i^{\thup}$ vertex of the $N$-gon.

Positive solutions of \eqref{SLMatEqInt} of total sum
$T = 3N-6$, the maximal value of~$T$,
are said to be \textit{totally positive}
(the reasons for this terminology are explained in~\cite{Val}).
Totally positive solutions give $-\Id$ in~\eqref{SLMatEqInt}.

\begin{thm}[Conway and Coxeter~\cite{CoCo}]
\label{CCthm}
The set of all totally positive solutions of \eqref{SLMatEqInt}
is equal to the set of all quiddities of triangulations of $N$-gons.
\end{thm}

To give a simple example, take $N=5$.
The triangulations of the pentagon are all rotations
of the one depicted here with its quiddity.
It is not difficult to show that for $N=5$ there exist exactly $5$ positive solutions of \eqref{SLMatEqInt}: the cyclic permutations of the quiddity $(1, 3, 1, 2, 2)$.
\begin{figure}[h]
\footnotesize
$$
\xymatrix @!0 @R=0.50cm @C=0.5cm
{
&& 3 \ar@{-}[rrd]\ar@{-}[lld]\ar@{-}[lddd]\ar@{-}[rddd]&
\\
1 \ar@{-}[rdd]&&&& 1 \ar@{-}[ldd]\\
\\
& 2 \ar@{-}[rr]&& 2
}
$$
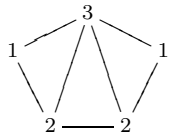
\normalsize
\end{figure}

It is classical that the number of triangulations of an $(n+2)$-gon is the \textit{Catalan number,}
\begin{equation} \label{C_n}
   C_n := \frac{1}{n+1}\binom{2n}{n}.
\end{equation}
Therefore the number of totally positive solutions of \eqref{SLMatEqInt} is $C_{N-2}$.

\subsection{Quiddities and 3-periodic dissections}
Recall that a \textit{dissection} of a convex $N$-gon is a partition
thereof into sub-polygons by non-crossing diagonals.
We will refer to these diagonals as the \textit{chords} of the dissection.
Just as for triangulations, the \textit{quiddity} of a dissection is the
cyclically ordered $N$-tuple $(a_1, \ldots, a_N)$, where $a_i$
is the number of sub-polygons contacting
the $i^{\thup}$ vertex of the $N$-gon.

We will rely on a combinatorial description of positive solutions
of \eqref{SLMatEqInt} which is one of the main results of~\cite{Val}.
It generalizes Theorem~\ref{CCthm} to what are referred to
in \cite{Val} as \textit{$3d$-dissections:}
dissections such that the number of vertices
of every sub-polygon is a multiple of $3$.
We shall modify this term to \textit{3-periodic dissections.}
The result is as follows.

\begin{thm}[\cite{Val}, Theorem~1.1(i)]
\label{ValRMSthm}
The set of all positive solutions of \eqref{SLMatEqInt} is equal to
the set of all quiddities of 3-periodic dissections of $N$-gons.
\end{thm}

For example, it is easy to verify directly that \eqref{SLMatEqInt}
has no positive solutions for $N=1$ or~$2$,
and a unique positive solution for $N=3$,
given by $(a_1, a_2, a_3) = (1, 1, 1)$, the quiddity
of the trivial dissection of the triangle.
The figure shows some simple examples of 3-periodic dissections
which are not triangulations, with their quiddities.
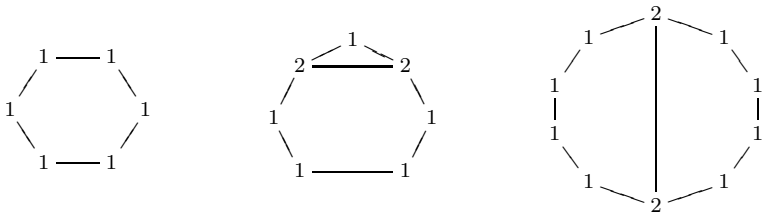
\begin{figure}[!htb]
\footnotesize
\begin{subfigure}{1.3in}
$$
\xymatrix @!0 @R=0.35cm @C=0.45cm
{
&1\ar@{-}[ldd]\ar@{-}[rr]&& 1\ar@{-}[rdd]&\\
\\
1\ar@{-}[rdd]&&&& 1\ar@{-}[ldd]\\
\\
&1\ar@{-}[rr]&&1
}
$$
\end{subfigure}
\begin{subfigure}{1.5in}
$$
\xymatrix @!0 @R=0.35cm @C=0.35cm
{
&&&1\ar@{-}[rrd]\ar@{-}[lld]
\\
&2\ar@{-}[ldd]\ar@{-}[rrrr]&&&& 2\ar@{-}[rdd]&\\
\\
1\ar@{-}[rdd]&&&&&& 1\ar@{-}[ldd]\\
\\
&1\ar@{-}[rrrr]&&&&1
}
$$
\end{subfigure}
\begin{subfigure}{1.6in}
$$
\xymatrix @!0 @R=0.32cm @C=0.45cm
 {
&&&{2}\ar@{-}[dddddddd]\ar@{-}[lld]\ar@{-}[rrd]&
\\
&{1}\ar@{-}[ldd]&&&& {1}\ar@{-}[rdd]\\
\\
{1}\ar@{-}[dd]&&&&&&{1}\ar@{-}[dd]\\
\\
1\ar@{-}[rdd]&&&&&&1\ar@{-}[ldd]\\
\\
&1\ar@{-}[rrd]&&&& 1\ar@{-}[lld]\\
&&&2&
}
$$
\end{subfigure}
\addtocounter{figure}{-1}
\normalsize
\end{figure}

By Theorem~\ref{ValRMSthm}, the number of 3-periodic dissections is
an upper bound for the number of positive solutions of~\eqref{SLMatEqInt}.
In fact, enumerating 3-periodic dissections is not difficult
and may be accomplished via standard combinatorial methods;
we present the result in Section~\ref{3dSec}.
There are many good sources for the techniques involved,
for example, the book \cite{Com} and the article \cite{FN}.

However, the upper bound thus obtained is not strict,
because, in contrast with triangulations,
3-periodic dissections are not determined by their quiddities.
The first occurrence of distinct 3-periodic dissections with the same quiddity
is the octagonal case shown.
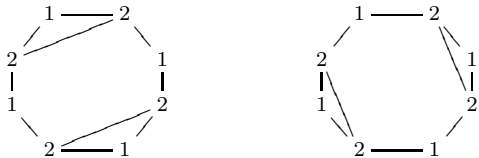
\begin{figure}[!htb]
\footnotesize
$$
\xymatrix @!0 @R=0.30cm @C=0.5cm
 {
&1\ar@{-}[ldd]\ar@{-}[rr]&& 2\ar@{-}[rdd]\\
\\
2\ar@{-}[dd]\ar@{-}[rrruu]&&&&1\ar@{-}[dd]\\
\\
1\ar@{-}[rdd]&&&&2\ar@{-}[ldd]\ar@{-}[llldd]\\
\\
&2\ar@{-}[rr]&& 1
}
\qquad\qquad\qquad
\xymatrix @!0 @R=0.30cm @C=0.5cm
 {
&1\ar@{-}[ldd]\ar@{-}[rr]&& 2\ar@{-}[rdd]\\
\\
2\ar@{-}[dd]\ar@{-}[rdddd]&&&&1\ar@{-}[dd]\\
\\
1\ar@{-}[rdd]&&&&2\ar@{-}[ldd]\ar@{-}[luuuu]\\
\\
&2\ar@{-}[rr]&& 1
}
$$
\normalsize
\end{figure}

Our main result is an exact count of the set of positive solutions of~\eqref{SLMatEqInt}.
In order to obtain it, we must enumerate the set of quiddities of 3-periodic dissections,
or in other words, the set of classes of 3-periodic dissections with the same quiddity.
Our approach is to construct a canonical representative of each class;
see Section~\ref{MaxOpenSec}.

\subsection{Remarks and a general problem} \label{RGP}
Although we succeed in counting the quiddities of 3-periodic dissections,
our method relies heavily on 3-periodicity.
It does not seem to adapt to arbitrary dissections,
and so we formulate the following general problem.
As far as we know it is open
and has not been considered in the literature.

\begin{pb}
\label{ThePb}
Count the number of quiddities of dissections.  More precisely, enumerate the distinct quiddities of the set of dissections of an $N$-gon into $m$~sub-polygons.
\end{pb}

We also mention some connections with other fields.
In Section~\ref{ProjPlane} we will see that certain
solutions of~\eqref{SLMatEqInt} correspond to
rational fans in~$\R^2$, relating the topic to the theory of toric surfaces.
Theorem~\ref{CCthm} was rediscovered in~\cite{Chr, Ste}
in this context, where it has become an important tool;
see~\cite{HTU} and references therein.
Another combinatorial model
was recently suggested in~\cite{CT},
encoding arbitrary solutions of~\eqref{SLMatEqInt}.

\subsection{Organization}
In Section~\ref{GenFunSec} we state our main results, characterizing the generating functions of the 3-periodic quiddities.  We give functional equations and formulas for the coefficients, solving the problem of counting positive solutions of~\eqref{SLMatEqInt}.

In Section~\ref{3dSec} we discuss the generating functions of the 3-periodic dissections themselves, as well as those of a more general family of classes of dissections.  This serves both as a review of the relevant techniques and as a source of information needed in the proofs of the main results.

Sections~\ref{LBi} and~\ref{Puiseux} complete the proofs of all results stated in Sections~\ref{GenFunSec} and~\ref{3dSec} except for Theorem~\ref{DiffQThm}, the functional equation satisfied by the bivariate generating function of the 3-periodic quiddities.  In Section~\ref{LBi}, Lagrange-B\"urmann inversion is used to determine the coefficients of the generating functions from their functional equations, and in Section~\ref{Puiseux}, asymptotic estimates are deduced from singularity analysis.  Both of these sections follow \cite{FN} closely.

Sections~\ref{MaxOpenSec} and~\ref{PrThmFMSec} are devoted to the proof of Theorem~\ref{DiffQThm}.  In Section~\ref{MaxOpenSec} we define the class of ``maximally open'' 3-periodic dissections and prove that it is in bijection with the set of 3-periodic quiddities.  This renders the enumeration of 3-periodic quiddities amenable to classical techniques, which we apply in Section~\ref{PrThmFMSec}.

We conclude in Section~\ref{ProjPlane} with an application of a special case of our main result to the enumeration of a certain class of toric varieties.  This special case was previously proven in~\cite{Gui}.

\subsection*{Acknowledgements}
We are grateful to Michel Brion, Michael Cuntz, Sophie Morier-Genoud,
Sergei Tabachnikov, and Sasha Voronov for enlightening discussions,
and we dedicate this paper to the memory of John Horton Conway.
Our discussions with John in October of 2013
were a crucial motivation for this work.
In particular, one of us asked him a na\"ive question:
``What is the reason for the connection between the
Conway-Coxeter equation and triangulations of an $N$-gon?'' 
His answer was ``There is no reason, it's a miracle!''
This ``miracle'' has intrigued and guided us for years.

\section{Quiddity generating functions}
\label{GenFunSec}

In this section we state our main results, descriptions of the \textit{univariate generating function} (UGF) and \textit{bivariate generating function} (BGF) enumerating the quiddities of 3-periodic dissections.  By Theorem~\ref{ValRMSthm}, these are also the generating functions enumerating the positive solutions of~\eqref{SLMatEqInt}.  We give functional equations and explicit formulas for the coefficients, and we give an asymptotic estimate related to work of V.~Kotesovec.  The proofs are developed in subsequent sections.

\subsection{The univariate generating function}
As noted in~\eqref{C_n}, the Catalan number $C_n$ is the
number of triangulations of an $(n+2)$-gon.
We will maintain this shift by~$2$ throughout the article.
Thus to translate between the integer $N$ of the
introduction and the integer $n$ below, set
\begin{equation*}
   N = n+2.
\end{equation*}

The Catalan generating function and its functional equation are
\begin{equation*}
   C(z) := \sum_{n=0}^\infty C_n z^n, \qquad
   C = 1 + z C^2.
\end{equation*}
The functional equation encodes the recursive formula for the coefficients, and Stirling's formula applied to~\eqref{C_n} gives an asymptotic estimate:
\begin{equation*}
   C_n = \sum_{i=0}^{n-1} C_i C_{n-1-i}, \qquad
   C_n = \frac{4^n}{\sqrt{\pi}\, n^{3/2}} \Big( 1 + O \Big(\frac{1}{n}\Big) \Big).
\end{equation*}
In this section we state the analogous results for 3-periodic quiddities.

\begin{defn}
Let $Q_n$ be the number of quiddities of 3-periodic dissections of $(n+2)$-gons, where by convention, $Q_0 := 1$.  The UGF of the 3-periodic quiddities is the formal power series
\begin{equation*}
   Q(z) := \sum_{n=0}^\infty Q_n z^n.
\end{equation*}
\end{defn}

In order to give the functional equation satisfied by $Q(z)$ we must introduce an auxiliary generating function $P(z)$.  Its combinatorial significance will be elucidated in Section~\ref{PrThmFMSec}.

\begin{defn}
Let $P(z) := \sum_{n=0}^\infty P_n z^n$ be the formal power series defined recursively by the equation
\begin{equation} \label{P(z) defn}
   P(z)
   = 1 + z P^2 + z^4 P^4 + z^7 P^6 + \cdots
   = 1 + \frac {z P^2}{1 - z^3 P^2} \,.
\end{equation}
\end{defn}

This formula determines the coefficients $P_n$: they are monotonically increasing positive integers which exceed the Catalan numbers for $n \ge 4$.  For $0 \le n \le 10$ they are
\begin{equation*}
   1,\, 1,\, 2,\, 5,\, 15,\, 48,\, 160,\, 550,\, 1937,\, 6954,\, 25355.
\end{equation*}
This sequence is known: it is a shift of A218251 in the Online Encyclopedia of Integer Sequences (OEIS) \cite{OEIS}, which was authored by P.~Hanna in 2012.  In 2013 V.~Kotesovec added the asymptotic estimate we restate below, as well as a degree~7 recurrence relation.

Our main results concerning $Q(z)$ and $P(z)$ are Theorems~\ref{TotQThm} and~\ref{Q_n}.  They are corollaries of their bivariate analogs, Theorems~\ref{DiffQThm} and~\ref{TheMainCountThm}.

\begin{thm} \label{TotQThm}
$Q(z)$ may be expressed as a rational function of $z$ and $P(z)$:
\begin{equation} \label{Q(z) defn}
   Q(z)
   = 1 + z P^2 + z^4 P^5 + z^7 P^8 + \cdots
   = 1 + \frac {z P^2}{1 - z^3 P^3} \,.
\end{equation}
\end{thm}

\begin{thm} \label{Q_n}
For $n > 0$, the coefficients of $P(z)$ and $Q(z)$ are
\begin{align*}
   & P_n = \sum_{0 \le k < n/3}
   \frac{1}{n-k+1}\, \binom{n - 2k - 1}{k} \binom{2n - 4k}{n - 3k},
   \\[6pt]
   & Q_n = \sum_{\substack{0 \le s \le k, \\ 0 \le k < n/3}}
   \frac{3(k-s)+2}{n-s+1}\, \binom{n-3k+s-2}{s} \binom{2n-3k-s -1}{n-3k-1}.
\end{align*}
\end{thm}

The sequence $Q_n$ is not yet in the OEIS.  Its initial terms are shown in the table.

\begin{table}[htbp]
\footnotesize
\begin{tabular}{|c||c|c|c|c|c|c|c|c|c|c|c|c|c|c|c|}
\hline
$n$ &0 & 1 & 2 & 3 & 4 & 5 & 6 & 7 & 8 & 9 & 10 &11 & 12 &13&14\\
\hline\hline
$Q_n$ & 1 & 1 & 2 & 5 & 15 & 49 & 166 & 577 & 2050 & 7414 & 27201 & 100984 & 378651 & 1431901 & 5454718 \\ 
\hline
\end{tabular}
\bigbreak
\caption*{The coefficients $Q_n$ of $Q(z)$: the number of positive solutions of~\eqref{SLMatEqInt}}
\normalsize
\end{table}

The final theorem of this section gives asymptotic estimates for $Q_n$ and $P_n$.  It will be proven in Section~\ref{Puiseux}, and for $P_n$ it is due to Kotesovec.  It involves the following positive algebraic numbers:
\begin{itemize}
\item
Let $\rho$ be the least positive root of the irreducible polynomial
\begin{equation} \label{Res_y F_P}
   4z^7 - 12z^5 - 8z^4 + 12z^3 - 20z^2 + 1.
\end{equation}
\smallbreak \item
Let $\nu$ be the least positive root of the irreducible polynomial
\begin{equation} \label{Res_z F_P}
   y^7 - 2y^6 - 4y^4 + 4y^3 + 4y -2.
\end{equation}
\bigbreak \item
Define\ 
$\displaystyle \gamma_P := \frac{1}{2}
\sqrt{\frac{\nu^3 - 2 \rho \nu^2 + 1}{\rho (3\rho \nu - \rho^2 + 1)}}$
\ and\ 
$\displaystyle \gamma_Q := \frac{\nu (\nu^3 + 2)}{(\nu^3 - 1)^2} \gamma_P$.
\end{itemize}
\medbreak \noindent
The approximate values of these numbers are
\begin{equation} \label{rngg}
   \rho \approx 0.237287, \quad
   \nu \approx 0.452578, \quad
   \gamma_P \approx 0.910244, \quad
   \gamma_Q \approx 1.047266.
\end{equation}

\begin{thm}[Kotesovec, A218251, \cite{OEIS}] \label{TotQasymp} \
\begin{enumerate}
\item[(i)]
$P(z)$ and $Q(z)$ both have radius of convergence $\rho$.
\smallbreak \item[(ii)]
At $\rho$, $P$ and $Q$ have infinite first derivatives but finite values:
\begin{equation*}
   P(\rho) = \nu / \rho, \qquad
   Q(\rho) = 1 + \nu^2 / \rho (1 - \nu^3).
\end{equation*}
\smallbreak \item[(iii)]
Asymptotically,
\begin{equation*}
   P_n = \frac{\gamma_P \rho^{-n}}{\sqrt{\pi}\, n^{3/2}}
   \Big( 1 + O \Big(\frac{1}{n}\Big) \Big),
   \qquad
   Q_n = \frac{\gamma_Q \rho^{-n}}{\sqrt{\pi}\, n^{3/2}}
   \Big( 1 + O \Big(\frac{1}{n}\Big) \Big).
\end{equation*}
\end{enumerate}
\end{thm}

\subsection{The bivariate generating function} \label{BGF Sec}
In this section we determine the number of positive solutions
of \eqref{SLMatEqInt} with a given value of the total sum $T$.

In any dissection of an $(n+2)$-gon into sub-polygons,
we refer to the sub-polygons as \textit{cells.}
Subpolygons with $R$ vertices are called $R$-cells.
We denote the total number of cells in the dissection by $m$.

Given a dissection of an $(n+2)$-gon with quiddity $(a_1, \ldots, a_{n+2})$, let us define the \textit{total sum} of the quiddity to be
\begin{equation*}
   T := a_1 + \cdots + a_{n+2}.
\end{equation*}
In light of Theorem~\ref{ValRMSthm}, this definition is compatible with~\eqref{MatEq Total Sum}.  The following lemma shows that for a given value of~$n$, fixing $m$ is equivalent to fixing~$T$.

\begin{lem} \label{T(m)}
Consider a dissection of an $(n+2)$-gon with $m$ cells.
\begin{enumerate}
\item[(i)]
The dissection has $m-1$ chords, and $T = n + 2m$.
\smallbreak \item[(ii)]
If the dissection is 3-periodic, then $m = n - 3k$ and $T = 3 (n - 2k)$ for some non-negative integer $k < \frac{1}{3} n$, the same $k$ appearing in~\eqref{AnyA}.
\end{enumerate}
\end{lem}

\begin{proof}
For~(i), use the fact that the number of chords contacting the $i^{\thup}$ vertex is $a_i - 1$.  For~(ii), note that any dissection can be made into a triangulation by triangulating each $R$-cell.  Because triangulating an $R$-cell adds $R - 3$ chords, the total number of new chords needed to convert a 3-periodic dissection to a triangulation is a multiple of~$3$.
\end{proof}

\begin{defn}
Let $Q_{n, m}$ be the number of quiddities of 3-periodic dissections of $(n+2)$-gons with $m$ cells, where by convention, $Q_{0, 0} := 1$ and $Q_{0, m} := 0$ for $m > 0$.  The BGF of the 3-periodic quiddities is the formal power series
\begin{equation} \label{Q defn}
   Q(z, w) := \sum_{n = 0}^\infty \sum_{m = 0}^\infty Q_{n, m} z^n w^m.
\end{equation}
\end{defn}

The following corollary of Theorem~\ref{ValRMSthm} and Lemma~\ref{T(m)} is immediate.

\begin{cor}
\begin{enumerate}
\item[(i)]
In any positive solution of \eqref{SLMatEqInt}, the total sum $T$ is $3 (n - 2k)$
for some non-negative integer $k < n/3$ (where $n = N - 2$).
\smallbreak \item[(ii)]
For $k < n/3$, the number of positive solutions of \eqref{SLMatEqInt}
with $T = 3 (n - 2k)$ is $Q_{n, n-3k}$.
\end{enumerate}
\end{cor}

As in the univariate case, in order to give the functional equation satisfied by $Q(z, w)$ we must introduce an auxiliary generating function $P(z, w)$, whose combinatorial significance will be given in Section~\ref{PrThmFMSec}.

\begin{defn}
Let $P(z, w) := \sum_{n=0}^\infty \sum_{m=0}^\infty P_{n, m} z^n w^m$ be the formal power series defined recursively by
\begin{equation} \label{P recurrence}
   P(z, w)
   = 1 + w z P^2 + w z^4 P^4 + w z^7 P^6 + \cdots
   = 1 + \frac {w z P^2}{1 - z^3 P^2} \,.
\end{equation}
\end{defn}

It is an abuse of notation to use the symbols $Q$ and $P$ for both the UGFs and the BGFs, but the arguments resolve the ambiguity.  Observe that evaluating the BGFs at $w = 1$ gives the UGFs:
\begin{equation} \label{QP U from B}
   Q(z) = Q(z, 1), \qquad P(z) = P(z, 1).
\end{equation}

We are now prepared to state the bivariate versions of Theorems~\ref{TotQThm} and~\ref{Q_n}: Theorems~\ref{DiffQThm} and~\ref{TheMainCountThm}, respectively.  They immediately imply their univariate counterparts: in light of~\eqref{QP U from B}, Theorem~\ref{TotQThm} is Theorem~\ref{DiffQThm} evaluated at $w = 1$ and Theorem~\ref{Q_n} is Theorem~\ref{TheMainCountThm} summed over~$k$.  

As discussed in the introduction, the proof of Theorem~\ref{DiffQThm} is of a different nature from the proofs of our other results.  It occupies Sections~\ref{MaxOpenSec} and~\ref{PrThmFMSec}.  Theorem~\ref{TheMainCountThm} follows from an application of Lagrange-B\"urmann inversion to \eqref{P recurrence} and~\eqref{Q recurrence}; the details are given in Section~\ref{LBi}.

\begin{thm} \label{DiffQThm}
$Q(z, w)$ is a rational function of $z$, $w$, and $P(z, w)$:
\begin{equation} \label{Q recurrence}
   Q(z, w)
   = 1 + w z P^2 + w z^4 P^5 + w z^7 P^8 + \cdots
   = 1 + \frac {w z P^2}{1 - z^3 P^3} \,.
\end{equation}
\end{thm}

\begin{thm} \label{TheMainCountThm}
For $n > 0$, the coefficients $P_{n, m}$ and $Q_{n, m}$ of $P(z, w)$ and $Q(z, w)$ are~$0$ unless $m = n - 3k$ for some non-negative integer $k < n/3$, in which case they are given by
\begin{align}
   & P_{n, n-3k} = \frac{1}{n-k+1}\, \binom{n-2k-1}{k} \binom{2n-4k}{n-3k},
   \label{Pnm formula}\\[6pt]
   & Q_{n, n-3k} = \sum_{0 \le s \le k}
   \frac{3(k-s)+2}{n-s+1}\, \binom{n-3k+s-2}{s} \binom{2n-3k-s -1}{n-3k-1}.
   \label{Qnm formula}
\end{align}
\end{thm}

Let us write $Q_{n, n-3k}$ explicitly at $k = 0$, $1$, and~$2$.
At $k = 0$, both formulas reduce to the Catalan numbers:
\begin{equation*}
   Q_{n, n} = P_{n, n} = C_n.
\end{equation*}
In fact, this may be seen without computation by letting $z$ go to~$0$ while holding $wz$ constant in \eqref{P recurrence} and~\eqref{Q recurrence}.

At $k = 1$ and~$2$ we obtain sequences not currently in the OEIS:
\begin{align}
\label{Qn3Eq}
   & Q_{n,n-3} =
   \binom{2n-4}{n-4} + \frac{6}{n+1} \binom{2n-5}{n-5} =
   \binom{2n-3}{n-4} - 2\binom{2n-5}{n-6}, \\[6pt]
\label{Qn6Eq}
   & Q_{n,n-6} =
   \frac{n-5}{2} \binom{2n-6}{n-7} - (n+2) \binom{2n-8}{n-9} - (n-2) \binom{2n-9}{n-10}.
\end{align}
The $k = 1$ sequence $Q_{n, n-3}$ plays a central role in our enumeration of blow-ups of the projective plane; see Section~\ref{ProjPlane}.  Although it is not in the OEIS, it is a sum of OEIS entries: the differential sequence $Q_{n, n-3} - P_{n, n-3}$ is A003517, and the sequence $P_{n, n-3}$ is A002694, a sequence of binomial coefficients with a number of combinatorial interpretations:
\begin{equation*}
    P_{n,n-3}=\binom{2n-4}{n-4}.
\end{equation*}

\begin{rem}
The $k=1$ sequence~\eqref{Qn3Eq} was calculated in Theorem~V.30 of \cite{Gui} using a different approach: the coefficient of $X^{k+6}$ in the formula for $G_{K_1}(X)$ given there is equal to $Q_{k+4, k+1}$.
\end{rem}

The tables here give the initial coefficients $P_{n, n-3k}$ and $Q_{n, n-3k}$, with $k$ fixed along rows.  They are accompanied by further remarks on $Q_{n, n-3k}$.

\begin{table}[htbp]
\footnotesize
\begin{tabular}{|c||c|c|c|c|c|c|c|c|c|c|c|c|c|c|c|}
\hline
$\;\;{}_{\textstyle k}\backslash{}{\textstyle n}$
& 0 & 1 & 2 & 3 & 4 & 5 & 6 & 7 & 8 & 9 & 10 &11 & 12 &13&14
\\
\hline\hline
\!0 & 1 & 1 & 2 & 5 & 14 & 42 & 132 & 429 & 1430 & 4862 
& 16796 & 58786 & 208012 & 742900 & 2674440
\\ 
\hline
\!1&  &&&& 1 & 6 & 28 & 120 & 495 & 2002 & 8008 & 31824 & 125970 & 497420 & 1961256
\\ 
\hline
\!2&  & &&&&&& 1& 12 & 90 & 550 & 3003 & 15288 & 74256 & 348840
\\ 
\hline
\!3&&&&&&&&&&& 1& 20 & 220 & 1820 & 12740
\\ 
\hline
\!4&&&&&&&&&&&&&& 1 & 30
\\ 
\hline
\end{tabular}
\bigbreak
\caption*{The coefficients $P_{n, n-3k}$ of $P(z,w)$}
\normalsize
\end{table}

\begin{table}[htbp]
\footnotesize
\begin{tabular}{|c||c|c|c|c|c|c|c|c|c|c|c|c|c|c|c|} 
\hline
$\;\;{}_{\textstyle k}\backslash{}{\textstyle n}$ &0 & 1 & 2 & 3 & 4 & 5 & 6 & 7 & 8 & 9 & 10 &11 & 12 &13&14\\
\hline\hline
\!0 & 1 & 1 & 2 & 5 & 14 & 42 & 132 & 429 & 1430 & 4862 & 16796 & 58786 & 208012 & 742900 & 2674440 \\ 
\hline
\!1&  &&&& 1 & 7 & 34 & 147 & 605 & 2431 & 9646 & 38012 & 149226 & 584630 & 2288132\\ 
\hline
\!2&  & &&&&&& 1& 15 & 121 & 758 & 4160 & 21098 & 101660 & 472872\\ 
\hline
\!3&&&&&&&&&&& 1& 26 & 315 & 2710 & 19234\\ 
\hline
\!4&&&&&&&&&&&&&& 1 & 40\\ 
\hline
\end{tabular}

\bigskip

\begin{tabular}{|c||c|c|c|c|c|c|c|c|c|c|c|} 
\hline
$\;\;{}_{\textstyle k}\backslash{}{\textstyle n}$ &15 & 16 & 17 & 18 & 19 & 20 &21 \\
\hline\hline
\!0 & 9694845 & 35357670 & 129644790 & 477638700 & 1767263190 & 6564120420 & 24466267020  \\ 
\hline
\!1& 8951945 & 35023365 & 137058495 & 536568150 & 2101610280 & 8235855870 &32292718290 \\ 
\hline
\!2& 2144397 & 9541895 & 41844935 & 181418250 & 779349480 & 3323000670 & 14081037000 \\ 
\hline
\!3&120887 & 699447 & 200720 & 19892125 & 100274020 & 492017955 & 2362240530\\ 
\hline
\!4& 680 & 7707 & 68875 & 527002 & 3617264 & 22924330 & 136717635 \\ 
\hline
\!5&  & 1 & 57 & 1295 & 18718 & 205953 & 1888162\\ 
\hline
\!6&  &&&& 1 & 77 & 2254\\ 
\hline
\end{tabular}
\bigbreak
\caption*{The coefficients $Q_{n, n-3k}$ of $Q(z,w)$}
\normalsize
\end{table}

\begin{itemize}

\item
The ``diagonal'' $Q_{3k+2, 2}$, beginning $2, 7, 15, 26, 40, 57, 77, 100, 126, \ldots$, is the second pentagonal number sequence, $\frac{1}{2} k' (3k' + 1)$, where $k' = k+1$.  It goes back to Euler; see OEIS~A005449.

\smallbreak\item
Each row $Q_{n, n-3k}$ grows faster than the previous one.
For instance, $Q_{n, n-3}$ dominates $Q_{n, n}$ starting from $n=17$.
We have the following bounds:
\begin{equation*}
 \frac{1}{k}\,\binom{2n-4k}{n-3k-1}\binom{n-2k-1}{k-1}\,
<\,Q_{n, n-3k}\, < \,
 \frac{1}{k}\,\binom{2n-3k}{n-3k-1}\binom{n-2k-1}{k-1}.
 \end{equation*}

\end{itemize}

\section{Dissection generating functions} \label{3dSec}

In preparation for the proofs of our main theorems we collect some results on dissections.  Section~\ref{MGF sec} concerns the \textit{multivariate generating function} (MGF) of the arbitrary dissections.  It provides a convenient tool for Section~\ref{Per sec}, which treats 3-periodic dissections, and more generally, \textit{$\ell$-periodic dissections.}  The results here are all either known or follow easily from well-known techniques, and in Section~\ref{History sec} we point out some of the relevant references.

\subsection{Arbitrary dissections} \label{MGF sec}
We employ the usual multi-index notation.  Write $\N$ for the non-negative integers and $\N^\omega$ for sequences $(m_1, m_2, \ldots)$ in $\N$ which are eventually zero.  Define
\begin{equation*}
   \vecm := (m_1, m_2, m_3, \ldots), \quad
   \vecw := (w_1, w_2, w_3, \ldots), \quad
   \vecw^{\vecm} := w_1^{m_1} w_2^{m_2} w_3^{m_3} \cdots.
\end{equation*}
Of course, for $\vecm \in \N^\omega$, $\vecw^{\vecm}$ is a finite product.  Set
\begin{equation*}   
   | \vecm | := \sum_{r = 1}^\infty m_r, \qquad
   \| \vecm \| := \sum_{r = 1}^\infty r m_r,
\end{equation*}
and use the multinomial coefficient expression
\begin{equation*}
   \binom{j + |\vecm |}{j,\ \vecm} :=
   \binom{j + |\vecm |}{j,\, m_1,\, m_2,\, m_3, \ldots}.
\end{equation*}

\begin{defn}
For $\vecm \in \N^\omega$, an \textit{$\vecm$-dissection} is a dissection such that for all positive integers~$r$, the number of $(r+2)$-cells is $m_r$.
\end{defn}

The reader may check that an $\vecm$-dissection is necessarily a dissection of an $(n+2)$-gon with $m$ cells, where
\begin{equation} \label{vecm nm}
   n = \|\vecm\|, \qquad m = |\vecm|.
\end{equation}

\begin{defn}
Let $D_{\vecm}$ be the number of $\vecm$-dissections, and set $D_{0, 0, 0, \ldots} := 1$.  The MGF of the dissections is the formal power series
\begin{equation} \label{Dissection MGF defn}
   D(\vecw) := \sum_{\vecm \in \N^\omega} D_{\vecm} \vecw^{\vecm}.
\end{equation}
\end{defn}

\begin{prop} \label{Dissection MGF}
The dissection MGF satisfies the recursive functional equation
\begin{equation} \label{DMFE}
   D(\vecw) = 1 + w_1 D(\vecw)^2 + w_2 D(\vecw)^3 + w_3 D(\vecw)^4 + \cdots.
\end{equation}
\end{prop}

\begin{proof}
This formula may be understood via a standard method; see Section~\ref{History sec}.  However, we have been unable to locate it in the literature, and in the course of the proof of our main results we will need a variation of the method.  For these reasons, we include a proof.

Fix some non-zero $\vecm$ and set $n := \|\vecm\|$.  Label the vertices of the $(n+2)$-gon by $0$ to $n+1$, in cyclic order.  Refer to the edge $(n+1, 0)$ as the \textit{base edge,} and in any dissection, refer to the cell containing the base edge as the \textit{base cell.}  The result will follow if we prove that the number of $\vecm$-dissections in which the base cell is an $(r+2)$-cell is equal to the coefficient of $\vecw^{\vecm}$ in $w_r D^{r+1}$.

Given such a dissection, label the vertices of the base cell by $v_0, \ldots, v_{r+1}$, where $0 = v_0 < v_1 < \cdots < v_r < v_{r+1} = n+1$, as in the figure.  For $0 \le s \le r$, consider the sub-dissection induced on the sub-polygon with vertices $v_s, v_s +1, \ldots, v_{s+1} -1, v_{s+1}$,
\begin{figure}[htbp]
\footnotesize
$$
\xymatrix @!0 @R=0.50cm @C=0.73cm
{&&&&&&&&& v_3
\\
&&&&&& \ar@{--}[llld] \circ \ar@{--}[rrr] &&& \bullet \ar@{--}[rrrd]
\\
&&& \circ \ar@{--}[lldd] &&&&&&&&& \circ \ar@{--}[rrdd]
\\
&&&&&&&&&&& \vecm(2)
\\
& \circ \ar@{--}[ldd] &&&&&&&&&&&&& \bullet
\ar@{--}[rdd] \ar@{-}[llllluuu] \ar@{-}[rdddd] & v_2
\\
\!\!\!\! v_{r-1}
\\
\bullet \ar@{--}[dd] \ar@{-}[rdddd] &&&&&&&
\mbox{\ base cell}
&&&&&&&
\mkern-9mu \vecm(1)
& \circ \ar@{--}[dd]
\\
\\
\circ \ar@{--}[rdd] &
\qquad\ \vecm(r-1)
&&&&&&&&&&&&&& \bullet \ar@{--}[ldd] \ar@{-}[llllllddddd]
\\
&&&&&&&&&&&&&&& v_1
\\
& \bullet \ar@{--}[rrdd] \ar@{-}[rrrrrddd] &&&&&&&&&&&
\mkern-27mu \vecm(0)
&& \bullet \ar@{-}[lldd]
\\
& v_r &&&
\vecm(r)
&&&&&&&&&& 2
\\
&&& \bullet \ar@{-}[rrrd] &&&&&&&&& \bullet \ar@{-}[llld]
\\
&&& n &&& \bullet\ar@{-}[rrr] \ar@<0.2mm>@{-}[rrr]
\ar@<-0.2mm>@{-}[rrr] &&& \bullet &&& 1
\\
&&&&&& v_{r+1} = n+1 &&& v_0=0
}
$$
\caption*{Counting dissections recursively
(hollow dots represent sub-dissections)}
\normalsize
\end{figure}
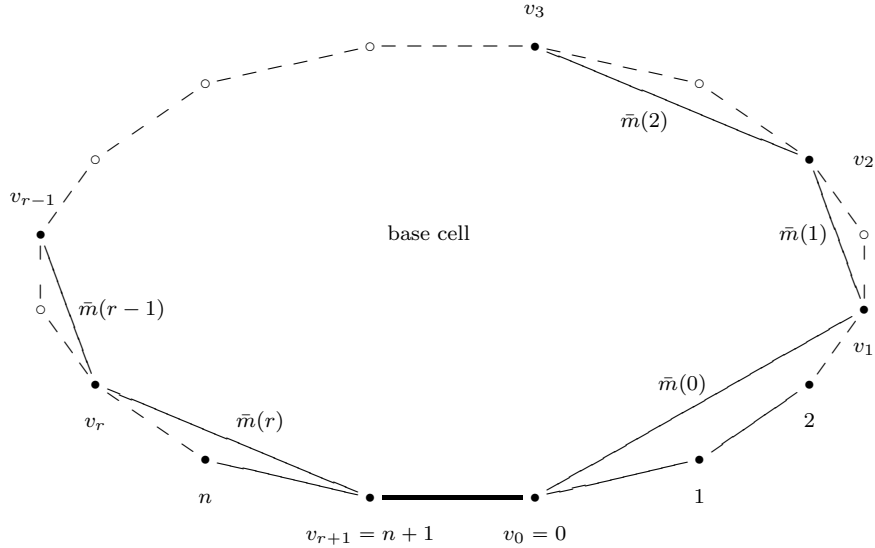
which is attached to the base cell along the chord $(v_s, v_{s+1})$.  It is an $\vecm(s)$-dissection, for some $\vecm(s)$ such that $\|\vecm(s)\| = v_{s+1} - v_s - 1$.  Note that if $v_{s+1} = v_s + 1$, then this sub-dissection is empty and $\vecm(s) = \bar 0$.

Because the entire dissection is an $\vecm$-dissection and the base cell accounts for one $r$-cell, we must have $\sum_0^r \vecm(s) = \vecm - e_r$, where $e_r$ denotes the standard basis vector in $\N^\omega$ with a $1$ in the $r^{\thup}$ spot and $0$'s elsewhere.  Observe that the coefficient of $\vecm - e_r$ in $D^{r+1}$ is the sum of all products $\prod_0^r D_{\vecm(s)}$, taken over all choices of $v_1, \ldots, v_r$, and for each such choice, over all choices of the $\vecm(s)$ such that
\begin{equation} \label{summation set}
   \|\vecm(s)\| = v_{s+1} - v_s - 1
   \quad \mbox{\rm and} \quad
   \sum_0^r \vecm(s) = \vecm - e_r.
\end{equation}
The proposition follows.
\end{proof}

The following result is well-known; see Section~\ref{History sec} for references.  We have included a brief proof in Section~\ref{LBi}, applying Lagrange-B\"urmann inversion to~\eqref{DMFE}.

\begin{thm} \label{D coefficients}
The coefficients of the MGF $D(\vecw)$ are
\begin{equation} \label{Dvm}
   D_{\vecm} = \frac{1}{\|\vecm\| + 1} \binom{\|\vecm\| + |\vecm|}{\|\vecm\|,\ \vecm}.
\end{equation}
\end{thm}

\subsection{Periodic dissections} \label{Per sec}
In addition to the generating functions of the 3-periodic dissections, we will need the generating functions of the ``odd dissections'', in which all cells have an odd number of vertices.  In order to present a unified treatment we consider a family of classes of dissections including both types.

\begin{defn}
Let $\ell$ be a positive integer.  A dissection is \textit{$\ell$-periodic} if each of its cells is a $(3 + \ell d)$-cell for some $d \in \N$.
\end{defn}

In particular, the 2-periodic dissections are the odd dissections, and 1-periodic dissections are simply arbitrary dissections.  Denote the UGF and BGF of the $\ell$-periodic dissections, respectively, by $D[\ell](z)$ and $D[\ell](z, w)$.  Thus
\begin{equation*}
   D[\ell](z) := \sum_{n=0}^\infty D[\ell]_n z^n, \qquad
   D[\ell](z, w) := \sum_{n, m = 0}^\infty D[\ell]_{n, m} z^n w^m,
\end{equation*}
where $D[\ell]_n$ is the number of $\ell$-periodic dissections of $(n+2)$-gons, and $D[\ell]_{n, m}$ is the number of such dissections which have $m$ cells.

\begin{prop} \label{Periodic BGF}
The BGF $D[\ell](z, w)$ satisfies
\begin{equation} \label{DlBFE}
   D[\ell](z, w)
   = 1 + wz D[\ell]^2 + wz^{1+\ell} D[\ell]^{2+\ell} + \cdots
   = 1 + \frac{wz D[\ell]^2}{1 - z^\ell D[\ell]^\ell} \,.
\end{equation}
\end{prop}

\begin{proof}
Say that a multi-index $\vecm$ is \textit{$\ell$-periodic} if $m_r = 0$ for all $r \not\equiv 1$ mod~$\ell$.  Clearly an $\vecm$-dissection is $\ell$-periodic if and only if $\vecm$ is $\ell$-periodic.    Combine this fact with \eqref{vecm nm} and~\eqref{Dissection MGF defn} to deduce that
\begin{equation} \label{Dlnm from Dvm}
   D[\ell]_{n, m} = \sum \big\{D_{\vecm}: \|\vecm\| = n,\ |\vecm| = m,\
   \vecm \mbox{\rm\ is $\ell$-periodic} \big\}.
\end{equation}

Substituting $w z^r$ for each factor $w_r$ in $\vecw^{\vecm}$ gives $z^{\|\vecm\|} w^{|\vecm|}$.  Therefore $D[\ell](z, w)$ is the series obtained from $D(\vecw)$ by substituting $w z^r$ for $w_r$ when $r \equiv 1$ mod~$\ell$, and $0$ for $w_r$ otherwise.  Hence the proposition follows from~\eqref{DMFE}.
\end{proof}

As in~\eqref{QP U from B}, the UGF is the BGF at $w=1$, and so we obtain:

\begin{cor}
The UGF $D[\ell](z)$ satisfies
\begin{equation} \label{DlUFE}
   D[\ell](z)
   = 1 + z D[\ell]^2 + z^{1+\ell} D[\ell]^{2+\ell} + \cdots
   = 1 + \frac{z D[\ell]^2}{1 - z^\ell D[\ell]^\ell} \,.
\end{equation}
\end{cor}

\begin{rems}
\begin{itemize}

\item
The quadratic formula gives a closed form of $D[1](z, w)$:
\begin{equation*}
   D[1](z, w) = \frac{ z + 1 - \sqrt{ z^2 - 2 (2w + 1) z + 1 }}{2 (w + 1) z}.
\end{equation*}

\smallbreak \item
The 3-periodic UGF $D[3](z)$ coincides with OEIS~A301832 up to order~$8$.

\end{itemize}
\end{rems}

\begin{thm} \label{Periodic KC}
For $0 < m \le n$, the coefficient $D[\ell]_{n, m}$ of $D[\ell](z, w)$ is~$0$ unless $m \equiv n$ mod~$\ell$, in which case it is
\begin{equation} \label{Dlnm}
   D[\ell]_{n, m} = \frac{1}{n+1} \binom{m-1 + (n-m)/\ell}{m-1} \binom{n+m}{m}.
\end{equation}
\end{thm}

This result will be proven in Section~\ref{LBi}.  However, let us make two remarks:

\begin{itemize}

\item
It is easy to see why $D[\ell](z, w)$ is~$0$ for $m \not\equiv n$ mod~$\ell$: it follows from the fact that if $\vecm$ is $\ell$-periodic, then $\|\vecm\| \equiv |\vecm|$ mod~$\ell$, because
\begin{equation*}
   \|\vecm\| - |\vecm| = \sum_r (r-1)m_r.
\end{equation*}
(Note that this generalizes Lemma~\ref{T(m)}(ii) from $D[3]$ to $D[\ell]$.)

\smallbreak \item
Combining \eqref{Dvm} and~\eqref{Dlnm from Dvm} gives $D[\ell]_{n, m}$ as a sum, but this sum does not imply \eqref{Dlnm} in any obvious way.

\end{itemize}

It is often convenient to reformulate \eqref{Dlnm} as follows: for $n > 0$, $D[\ell]_{n, m} = 0$ unless $m = n - \ell k$ for some non-negative integer $k < n/\ell$, and
\begin{equation} \label{alt periodic KC}
   D[\ell]_{n, n - \ell k} = \frac{1}{n+1} \binom{n - (\ell-1)k - 1}{k} \binom{2n - \ell k}{n - \ell k}.
\end{equation}

Applying $D[\ell](z) = D[\ell](z, 1)$ gives a formula for $D[\ell]_n$:

\begin{cor}
For $n > 0$, the coefficients $D[\ell]_n$ of the UGF $D[\ell](z)$ are
\begin{equation*}
   D[\ell]_n = \frac{1}{n+1} \sum_{0 \le k < n/\ell}
   \binom{n - (\ell-1)k - 1}{k} \binom{2n - \ell k}{n - \ell k}.
\end{equation*}
\end{cor}

We will discuss $D[\ell](z)$ from an analytic standpoint in Section~\ref{Puiseux}.

\begin{rem}
Because $Q$ counts 3-periodic quiddities and $D[3]$ counts 3-periodic dissections, it is clear that the coefficients of $Q$ are majorized by those of $D[3]$.  We will see in Section~\ref{PrThmFMSec} that $P$ counts only certain 3-periodic quiddities, so its coefficients are majorized by those of $Q$.  An examination of \eqref{alt periodic KC} shows that $D[\ell]_{n, n - \ell k}$ decreases as $\ell$ increases, and comparison with \eqref{Pnm formula} shows that $P_{n, n-3k}$ majorizes $D[4]_{n, n-4k}$.  Thus we have
\begin{equation*}
   D[1]_{n, n-k} \ge D[2]_{n, n-2k} \ge D[3]_{n, n-3k} 
   \ge Q_{n, n-3k} \ge P_{n, n-3k} \ge D[4]_{n, n-4k} \ge \cdots,
\end{equation*}
\begin{equation} \label{coefficient comparison}
   D[1]_n \ge D[2]_n \ge D[3]_n 
   \ge Q_n \ge P_n \ge D[4]_n \ge \cdots \ge C_n.
\end{equation}
\end{rem}

\subsection{Historical remarks} \label{History sec}

The pictorial argument we have used to prove \eqref{DMFE} is a special case of the \textit{symbolic enumeration method;} see for example Section~0.1 of \cite{FN}.  In Section~3.1 of the same paper the authors use it to give a derivation of \eqref{DlBFE} for $\ell = 1$; our argument for \eqref{DMFE} is essentially the same.  See Section~7.1 of \cite{Ben} for another relevant example, giving the UGF of the dissections such that the number of sides of each cell lies in any prescribed subset of $\{3, 4, 5, \ldots\}$.

A proof of \eqref{Dvm} by bijection may be found in Corollary~4.2 of \cite{Gai}.  The formula also appears in \cite{GJ}, as Exercise~2.7.14, and in \cite{DR}.

For $\ell = 1$, \eqref{Dlnm} is known as the \textit{Kirkman-Cayley formula,} as it was conjectured in~\cite{Kir} and proven in~\cite{Cay}.  (It was also stated as a question in~\cite{Pro}.)  Proofs using generating functions may be found in \cite{FN, Rea}, and proofs by bijection are given in \cite{Gai, PS, Sta96}.  Our proof for arbitrary~$\ell$ is a straightforward generalization of the argument given in \cite{FN}.

\section{Lagrange-B\"urmann inversion} \label{LBi}

We now apply Lagrange-B\"urmann inversion to prove Theorems~\ref{TheMainCountThm}, \ref{D coefficients}, and~\ref{Periodic KC} (in reverse order).  Suppose that $\phi(u)$ is a formal power series in $u$ with a non-zero constant term.  Then there is clearly a unique formal series solution $y(z)$ of the functional equation $y(z) = z \big(\phi \circ y(z)\big)$.  Lagrange-B\"urmann inversion gives the coefficients of $y(z)$.  More generally, if $\psi(u)$ is any formal series, it gives the coefficients of $\psi \circ y(z)$.  The result is
\begin{equation} \label{LB}
   (n+1) [z^{n+1}] (\psi \circ y) = [u^n] (\psi' \phi^{n+1}),
\end{equation}
where $[x^i]f$ denotes the coefficient of $x^i$ in a formal series $f(x)$.

This is a well-known classical theorem; for further discussion and references, see \cite{Com, FN}.  Let us briefly outline the proof.  Because $[z^{n+1}] (\psi \circ y)$ depends only on the initial terms of $\phi$ and $\psi$, we may take them to be polynomial.  This gives
\begin{equation*}
   2 \pi i (n+1) [z^{n+1}] (\psi \circ y) =
   2 \pi i [z^n] (\psi \circ y)' =
   {\oint}_{\!\!0} \frac{\psi' \circ y}{z^{n+1}} y' dz.
\end{equation*}
The lowest non-zero term of the series $y(z)$ is linear, so we may apply the change of variables $u = y(z)$.  Combine this with the fact that $z = y / (\phi \circ y)$ and continue the above equation as follows to complete the proof:
\begin{equation*}
   {\oint}_{\!\!0} \frac{\psi'(u)}{z^{n+1}} du =
   {\oint}_{\!\!0} \frac{\psi' \phi^{n+1}}{u^{n+1}} du =
   2 \pi i [u^n] (\psi' \phi^{n+1}).
\end{equation*}

We will frequently need a special case of~\eqref{LB}: take $\psi(u) = u^e$ and substitute $n+e-1$ for $n$ to obtain
\begin{equation} \label{LBe}
   (n+e) [z^{n+e}] y^e = e [u^n] \phi^{n+e}.
\end{equation}

\begin{prop} \label{Gen Per KC}
For any positive integer~$e$,
\begin{equation*}
   D[\ell](z, w)^e = 1 + \sum_{\substack{n,\, k: \\ 0 \le k < n/\ell}}
   \frac{e}{n+e} \binom{n - (\ell-1)k - 1}{k} \binom{2n - \ell k + e - 1}{n - \ell k}
   z^n w^{n-\ell k}.
\end{equation*}
\end{prop}

\begin{proof}
It suffices to prove the following generalization of~\eqref{alt periodic KC}: for $n > 0$,
\begin{equation*}
   (n+e) [z^n] D[\ell]^e(z, w) = e \sum_{0 \le k < n/\ell}
   \binom{n - (\ell-1)k - 1}{k} \binom{2n - \ell k + e - 1}{n - \ell k} w^{n - \ell k}.
\end{equation*}

Let $y_\ell(z, w)$ be the shifted BGF $z D[\ell](z, w)$.  Regarding $w$ as a parameter, multiply \eqref{DlBFE} by $z$ and rearrange to obtain $y_\ell = z (\phi_\ell \circ y_\ell)$, where
\begin{equation*}
   \phi_\ell(u) := \left( 1 - \frac{w u}{1 - u^\ell} \right)^{-1}.
\end{equation*}
Applying~\eqref{LBe}, we have
\begin{equation*}
   (n+e) [z^n] D[\ell]^e(z, w) =
   (n+e) [z^{n+e}] y_\ell^e(z, w) =
   e [u^n] \phi_\ell^{n+e}(u, w).
\end{equation*}
Use $(1-x)^{-(n+e)} = \sum_{i=0}^\infty \binom{n+e+i-1}{i} x^i$ and the assumption $n > 0$ to obtain
\begin{align*}
   & [u^n] \left( 1 - \frac{w u}{1 - u^\ell} \right)^{-(n+e)} =\
   [u^n] \sum_{i=1}^\infty \binom{n+e+i-1}{i} \left( \frac{w u}{1 - u^\ell} \right)^i \\[6pt]
   & =\ \sum_{i=1}^\infty \binom{n+e+i-1}{i} w^i [u^{n-i}] (1 - u^\ell)^{-i} \\[6pt]
   & =\ \sum_{i=1}^\infty \sum_{k=0}^\infty \binom{n+e+i-1}{i} \binom{i+k-1}{k} w^i [u^{n-i}] u^{\ell k}.
\end{align*}
The sum over $i$ contributes only at $i = n - \ell k$, completing the proof.
\end{proof}

\medbreak \noindent
\textit{Proof of Theorem~\ref{Periodic KC}.}
Apply Proposition~\ref{Gen Per KC} at $e = 1$.

\medbreak \noindent
\textit{Proof of Theorem~\ref{D coefficients}.}
In order to use~\eqref{LB}, we substitute $w_r := v_r z^r$ in $D(\vecw)$.  Then~\eqref{DMFE} becomes
\begin{equation*}
   D(\vecw) = 1 + v_1 z D^2 + v_2 z^2 D^3 + \cdots,
\end{equation*}
and~\eqref{Dvm} is equivalent to
\begin{equation*}
   (n+1) [z^n] D(\vecw) = \sum_{\{\vecm:\, \|\vecm\| = n\}}
   \binom{n + |\vecm|}{n,\ \vecm} \vecv^{\vecm}.
\end{equation*}

Define $y_D(z, \vecv) := z D(\vecw)$ and check that
\begin{equation*}
   y_D(z, \vecv) = z (1 - v_1 y_D - v_2 y_D^2 - v_3 y_D^3 - \cdots)^{-1}.
\end{equation*}
Thus $y_D(z, \vecv) = z \big(\phi_D \circ y_D(z, \vecv)\big)$, where $\phi_D(u, \vecv) := (1 - v_1 u - v_2 u^2 - \cdots)^{-1}$, a well-defined formal series in~$u$.

Apply~\eqref{LBe} with $e=1$:
\begin{equation*}
   (n+1) [z^n] D(\vecw) =
   (n+1) [z^{n+1}] y_D(z, \vecv) =
   [u^n] \phi_D^{n+1}(u, \vecv).
\end{equation*}
Following the argument used for Proposition~\ref{Gen Per KC}, this becomes
\begin{equation*}
   [u^n] (1 - v_1 u - v_2 u^2 - \cdots)^{-(n+1)} =
   \sum_{i=0}^\infty \binom{n+i}{i} [u^n] (v_1 u + v_2 u^2 + \cdots)^i.
\end{equation*}
To complete the proof, note that
\begin{equation*}
   [u^n] (v_1 u + v_2 u^2 + \cdots)^i =
   \sum_{\{\vecm:\, \|\vecm\| = n,\, |\vecm| = i\}}
   \binom{i}{\vecm} \vecv^{\vecm}
\end{equation*}
and $\binom{n+i}{i} \binom{i}{\vecm} = \binom{n+i}{n,\, \vecm}$.
\hfill $\Box$

\medbreak
The properties of $P(z, w)$ needed to prove Theorem~\ref{TheMainCountThm} do not follow directly from~\eqref{LB}, but rather from Proposition~\ref{Gen Per KC} for $D[2]$ combined with a certain relationship between $P$ and $D[2]$.  Define $\tilde D[2](z, w)$ by
\begin{equation*}
   \tilde D[2](z, w) := D[2](z^{3/2}, wz^{-1/2}),
\end{equation*}
and use Proposition~\ref{Gen Per KC} to see that it is a formal series in $z$ and $w$ with non-negative integral exponents.  We begin with two preparatory lemmas.

\begin{lem} \label{tilde D2}
$P(z, w) = \tilde D[2](z, w)$.
\end{lem}

\begin{proof}
By~\eqref{DlBFE}, $D[2](z, w) = 1 + wzD[2]^2 / (1 - z^2 D[2]^2)$, and so $\tilde D[2]$ satisfies
\begin{equation*}
   \tilde D[2](z, w) = 1 + wz \tilde D[2]^2 / (1 - z^3 \tilde D[2]^2).
\end{equation*}
This is the same recursive functional equation \eqref{P recurrence} defining $P(z, w)$.
\end{proof}

\begin{lem} \label{Pe}
For any positive integer~$e$,
\begin{equation*}
   P^e(z, w) = 1 + \sum_{\substack{n,\, k: \\ 0 \le k < n/3}}
   \frac{e}{n-k+e} \binom{n - 2k - 1}{k} \binom{2n - 4k + e - 1}{n - 3k}
   z^n w^{n-3k}.
\end{equation*}
\end{lem}

\begin{proof}
By Proposition~\ref{Gen Per KC} and the definition of $\tilde D[2](z, w)$,
\begin{equation*}
   \tilde D[2]^e = 1 +
   \sum_{\substack{\tilde n,\, \tilde k: \\ 0 \le \tilde k < \tilde n/2}}
   \frac{e}{\tilde n+e} \binom{\tilde n - \tilde k - 1}{\tilde k}
   \binom{2 \tilde n - 2 \tilde k + e - 1}{\tilde n - 2 \tilde k}
   (z^{3/2})^{\tilde n} (w z^{-1/2})^{\tilde n - 2 \tilde k}.
\end{equation*}
Hence the lemma follows from Lemma~\ref{tilde D2} and the substitution $n = \tilde n + \tilde k$, $k = \tilde k$.
\end{proof}

\medbreak \noindent
\textit{Proof of Theorem~\ref{TheMainCountThm}.}
For~\eqref{Pnm formula}, apply Lemma~\ref{Pe} at $e=1$.  For~\eqref{Qnm formula}, restate Lemma~\ref{Pe} as follows: for $0 \le 3k' < n'$, the coefficient of $z^{n'} w^{n'-3k'}$ in $P^{e'}$ is
\begin{equation*}
   (P^{e'})_{n', n'-3k'} = \frac{e'}{n'-k'+e'}
   \binom{n' - 2k' - 1}{k'} \binom{2n' - 4k' + e' - 1}{n' - 3k'}.
\end{equation*}

Now apply Theorem~\ref{DiffQThm} to obtain
\begin{equation*}
   Q_{n, n-3k} = \sum_{j=0}^k (P^{3j+2})_{n-3j-1, n-3k-1}
\end{equation*}
for $n > 0$.  Substituting $n' = n-3j-1$, $k' = k-j$, and $e' = 3j+2$ in the formula for $(P^{e'})_{n', n'-3k'}$ and then replacing $k-j$ by $s$ completes the proof.
\hfill $\Box$

\section{Asymptotic estimates} \label{Puiseux}

Here we use a classical strategy presented in Section~4 of \cite{FN} to prove Theorem~\ref{TotQasymp} and give a conjectural asymptotic estimate of the coefficients of the periodic dissection UGF $D[\ell](z)$.  The conjecture depends on the distribution of the roots of a certain polynomial of degree~$2\ell$ and may be checked with software for any particular~$\ell$; we have verified it for $\ell \le 16$.

Suppose that $F(z, y)$ is a real polynomial such that
\begin{equation*}
   F(0, 0) = 0, \qquad \partial_y F(0, 0) \not= 0.
\end{equation*}
Let $y(z)$ be the branch of the graph of $F(z, y) = 0$ passing through the origin, i.e., the analytic function such that $y(0) = 0$ and $F(z, y(z)) = 0$.  Recall that
\begin{equation} \label{y prime}
   y'(z) = -(\partial_z F / \partial_y F)|_{(z, y(z))}.
\end{equation}

\begin{thm} \label{Darboux}
Let $\sum_{n=1}^\infty b_n z^n$ be the Maclaurin series of $y(z)$.  Make the following assumptions concerning it:
\begin{enumerate}

\item[(i)]
The coefficients $b_n$ are non-negative real numbers.

\smallbreak \item[(ii)]
The radius of convergence of the series is $\rho < \infty$.

\smallbreak \item[(iii)]
$\rho$ is the unique singularity of $y(z)$ of magnitude $\rho$.

\smallbreak \item[(iv)]
$\lim_{z \to \rho^-} y = \nu < \infty$.

\smallbreak \item[(v)]
$\partial_z F(\rho, \nu) \not= 0$ and $\partial_y^2 F(\rho, \nu) \not= 0$.

\end{enumerate}

Then asymptotically,
$\displaystyle b_n = \frac{\gamma \rho^{-n}}{\sqrt{\pi}\, n^{3/2}}
  \Big( 1 + O \Big(\frac{1}{n}\Big) \Big)$,
where
$\displaystyle \gamma = \sqrt{ \frac{\rho\, \partial_z F}
  {2\, \partial_y^2 F}} \,\Bigg|_{(\rho, \nu)} > 0$.
\end{thm}

\begin{proof}
We will only outline the proof; further details and historical references are given in \cite{FN}.  See also Theorem~5 of \cite{Ben} for a similar result.

The idea is to expand the inverse function $z(y)$ at $(\rho, \nu)$.  Because $y(z)$ is a non-negative series with radius of convergence $\rho$, it must be singular at $z = \rho$, which implies that $\partial_y F(\rho, \nu) = 0$.

Because $\partial_z F(\rho, \nu) \not= 0$, the branch of the graph of $F(z, y) = 0$ passing through $(\rho, \nu)$ may be regarded as an analytic function $z(y)$.  Check that
\begin{equation*}
   z''(\nu) = -(\partial_y^2 F / \partial_z F)|_{(\rho, \nu)} = -\rho / 2 \gamma^2.
\end{equation*}
To see that $\gamma$ may be taken real and positive, note that $y(z)$ and all its derivatives are non-negative on $[0, \rho)$, so $z''(y)$ must be non-positive at $\nu$.

Conclude that $z(y)$ may be expressed as $\sum_{m=0}^\infty \beta_m (y - \nu)^m$, where
\begin{equation*}
   \beta_0 = \rho, \qquad \beta_1 = 0, \qquad
   \beta_2 = -\rho / 4 \gamma^2.
\end{equation*}
This can be written as
\begin{equation*}
   4 \gamma^2 (1 - z/\rho) = (y - \nu)^2
   \big(1 + \beta'_3 (y - \nu) + \beta'_4 (y - \nu)^2 + \cdots\big),
\end{equation*}
where $\beta'_3 = -4 \gamma^2 \beta_3 / \rho$, $\beta'_4 = -4 \gamma^2 \beta_4 / \rho$, and so on.

Take the square root.  Because $y < \nu$ for $z < \rho$, we obtain a series of the form
\begin{equation*}
   2 \gamma (1 - z/\rho)^{1/2} = (\nu - y)
   \big(1 + \beta''_3 (\nu - y) + \beta''_4 (\nu - y)^2 + \cdots\big),
\end{equation*}
where $\beta''_3, \beta''_4, \ldots$, are scalars beginning with $\beta''_3 = -\beta'_3/2 = 2 \gamma^2 \beta_3 / \rho$.

Invert this series algebraically to obtain a series
\begin{equation} \label{numy}
   \nu - y = c_1 (1 - z/\rho)^{1/2} + c_2 (1 - z/\rho) + c_3 (1 - z/\rho)^{3/2} + \cdots
\end{equation}
for some scalars $c_i$ beginning with $c_1 = 2 \gamma$.  The theorem now follows from
\begin{equation*}
   [z^n] (1 - z)^{1/2} = \frac{1}{2 \sqrt{\pi}\, n^{3/2}}
  \Big( 1 + O \Big(\frac{1}{n}\Big) \Big),
\end{equation*}
coupled with the classical method of \textit{asymptotic transfer} and the fact that by assumption, $\rho$ is the unique dominant singularity of $y(z)$ with respect to $z = 0$.
\end{proof}

\medbreak \noindent
\textit{Proof of Theorem~\ref{TotQasymp}.}
Define $y_P(z) := z P(z)$, so that \eqref{P(z) defn} becomes
\begin{equation} \label{y_P}
   y_P(z)
   = z + y_P^2 + z y_P^4 + z^2 y_P^6 + \cdots
   = z + \frac{y_P^2}{1 - z y_P^2} \,.
\end{equation}
Thus $F_P \big(z, y_P(z) \big) = 0$ for all $z$ in the domain of $y_P$, where
\begin{equation*}
   F_P(z, y) := z y^3 + (1 - z^2) y^2 - y + z.
\end{equation*}
Because $F_P(0, 0) = 0$ and $\partial_y F_P(0,0) \not= 0$, we may proceed to check the hypotheses (i)-(v) of Theorem~\ref{Darboux}.

Regarding \eqref{y_P} as a recursive functional equation, it is clear that the Maclaurin series of $y_P$ at $z = 0$ has integer coefficients exceeding the Catalan numbers (this also follows from Theorem~\ref{Q_n}).  Therefore the radius of convergence $\rho$ of $y_P$ is at most $1/4$, and so (i) and~(ii) are proven.

Because it has a positive series, $y_P$ must be singular at~$\rho$.  By~\eqref{y prime}, $\partial_y F_P$ is zero at singularities of $y_P$.  As a polynomial in~$y$, $F_P$ has lead coefficient~$z$.  Since $y_P$ is not singular at $z = 0$, it can only be singular at roots of the discriminant of $F_P$ regarded as a cubic in~$y$.  This discriminant is \eqref{Res_y F_P}, which has positive roots near $0.24$ and $1.97$, a negative root near $-0.21$, and two pairs of complex roots, of magnitudes near $0.94$ and $1.70$.  Since $\rho < 1/4$, it must be the least positive root, as given in~\eqref{rngg}.  Moreover, (iii) is proven.

At this point (iv) is also proven, as $\nu$ must be the simultaneous root of $F(\rho, y)$ and $\partial_y F_P(\rho, y)$.  To express it as a root of a rational polynomial, solve $F_P(z, y_P) = 0$ for $z$ to obtain the inverse $z_P(y)$ of $y_P(z)$.  The fact that $y_P(0) = 0$ determines which branch of the square root to take:
\begin{equation*}
   z_P(y) = \frac{(y^3 + 1) - \sqrt{(y^3 - 1)^2 + 4 y^4}}{2 y^2}.
\end{equation*}

Calculus shows that $z_P'$ can only be zero at roots of the polynomial \eqref{Res_z F_P}, which has three real roots, near $0.45$, $1.13$, and $2.37$.  Because $z_P = 0$ at $y = 0, 1$, the root near $0.45$ is its least positive (and in fact only) maximum.  Because $y_P$ increases monotonically on $[0, \rho)$, this root must be $\nu$, as given in~\eqref{rngg}.

For~(v), note that
\begin{equation*}
   \partial_z F_P(\rho, \nu) = \nu^3 - 2 \rho \nu^2 + 1, \qquad
   \partial_y^2 F_P(\rho, \nu) = 2(3 \rho \nu - \rho^2 + 1).
\end{equation*}
The statements of Theorem~\ref{TotQasymp} concerning $P(z)$ now follow from Theorem~\ref{Darboux}.

For the statements concerning $Q$, set $y_Q(z) := z Q(z)$ and $g(y) := y^2 / (1 - y^3)$.  Then \eqref{Q(z) defn} becomes $y_Q = z + g(y_P)$, so $y_Q$ and $Q$ have the same radius of convergence $\rho$ as $y_P$ and $P$, and their values at $\rho$ are as claimed.

For the asymptotic estimate of $Q_n$, note that asymptotically the coefficients of $y_Q$ and $g(y_P)$ are the same.  Combining
\begin{equation*}
   g(\nu) - g(y) = g'(\nu) (\nu - y) + O(\nu - y)^2
\end{equation*}
with \eqref{numy} leads to $\gamma_Q = g'(\nu) \gamma_P$, completing the proof.
\hfill $\Box$

\begin{rems}
\begin{itemize}

\item
The discriminant \eqref{Res_y F_P} of $F_P$ has a real root $\mu$ near $-0.21$, of magnitude less than~$\rho$.  We know that $y_P(z)$ is analytic at $\mu$, so it must be that $F_P(\mu, y) = 0$ has two distinct roots in~$y$, one double and one simple, and $y_P(\mu)$ is the simple one.

\smallbreak \item
The polynomial \eqref{Res_z F_P} is the resultant of $F_P$ and $\partial_y F_P$ regarded as quadratics in~$z$, divided by~$y$.  Not all of its real roots are extrema of $z_P$: the roots near $0.45$ and $2.37$ are, but the root near $1.13$ is the minimum of the conjugate of $z_P$, in which the other branch of the square root is taken.

\end{itemize}
\end{rems}

We conclude this section with some observations on the analytic behaviour of the UGF $D[\ell](z)$.  Consider the shift $y_\ell(z) := z D[\ell](z)$.  By~\eqref{DlUFE},
\begin{equation} \label{y_l}
   y_\ell(z)
   = z + y_\ell^2 + y_\ell^{\ell + 2} + y_\ell^{2\ell + 2} + \cdots
   = z + \frac{y_\ell^2}{1 - y_\ell^\ell} \,.
\end{equation}
Thus $F_\ell \big(z, y_\ell(z) \big) = 0$ for all $z$ in the domain of $y_\ell$, where
\begin{equation} \label{F_l}
   F_\ell(z, y) := y^{\ell + 1} - z y^\ell + y^2 - y + z.
\end{equation}

At $(0, 0)$, $F_\ell = 0$ and $\partial_y F_\ell \not= 0$, so Theorem~\ref{Darboux} will give asymptotic estimates for the coefficients of the Maclaurin series of $y_\ell$ if its hypotheses (i)-(v) hold.  However, for general~$\ell$ we have only been able to verify (i), (ii), (iv), and~(v).  Indeed, by either Corollary~\ref{alt periodic KC} or direct argument from \eqref{y_l}, the series of $y_\ell$ has positive coefficients exceeding the Catalan numbers, so it has radius of convergence $\le 1/4$.  This gives (i) and~(ii), and (iv) and~(v) then follow from~\eqref{F_l}.

In order to describe $y_\ell(z)$ we define the following positive algebraic numbers:

\begin{itemize}
\item
Let $\nu_\ell$ be the least positive root of the polynomial
\begin{equation*}
   R_\ell(y) := y^{2\ell} - (\ell-2) y^{\ell+1} -2y^\ell - 2y + 1.
\end{equation*}
\medbreak \item
Set $\rho_\ell := z_\ell(\nu_\ell)$, where $z_\ell(y)$ is the inverse of $y_\ell(z)$:
\begin{equation*}
   z_\ell(y) := y - \frac{y^2}{1 - y^\ell} =
   \frac{y (y^\ell + y - 1)}{y^\ell - 1} =
   y - y^2 - y^{\ell + 2} - y^{2\ell + 2} - \cdots.
\end{equation*}
\end{itemize}

\begin{prop}
\begin{enumerate}
\item[(i)]
$y_\ell(z)$ and $D[\ell](z)$ have radius of convergence $\rho_\ell$.
\smallbreak \item[(ii)]
At $z = \rho_\ell$, $y_\ell$ and $D[\ell]$ have infinite first derivative but finite value:
\begin{equation*}
   y_\ell(\rho_\ell) = \nu_\ell, \qquad
   D[\ell](\rho_\ell) = \frac{\nu_\ell}{\rho_\ell} =
   \frac{1 - \nu_\ell^\ell}{1 - \nu_\ell - \nu_\ell^\ell}.
\end{equation*}
\end{enumerate}
\end{prop}

\begin{proof}
The inverse $z_\ell(y)$ of $y_\ell(z)$ increases monotonically on the interval $(0, \nu_\ell)$ to a local maximum of value $\rho_\ell$, as
\begin{equation*}
   z_\ell'(y) = \frac{R_\ell(y)}{(y^\ell - 1)^2}
   = 1 - 2y - (\ell + 2) y^{\ell + 1} - (2\ell + 2) y^{2\ell + 1} - \cdots.
\end{equation*}
Hence the proposition follows from the fact that $y_\ell(z)$ has a positive series and is non-singular on the interval $(0, \rho_\ell)$.
\end{proof}

\begin{pb}
Does $y_\ell$ satisfy (iii) of Theorem~\ref{Darboux}?  In other words, is $z = \rho_\ell$ the unique singularity of $y_\ell$ of magnitude $\rho_\ell$?  If so, the theorem would give the asymptotic estimate
\begin{align} \label{D_l asymp}
   & D[\ell]_n = \frac{\gamma_\ell \rho_\ell^{-n}}{\sqrt{\pi}\, n^{3/2}}
   \Big( 1 + O \Big(\frac{1}{n}\Big) \Big),
   \mbox{\rm\ \ where\ \ } \\[6pt] \nonumber
   & \gamma_\ell := \frac{1 - \nu_\ell^\ell}{\sqrt{2 \rho_\ell}}\,
   \Big( 2 + 2\ell \nu_\ell^{\ell-1} +(\ell-2)(\ell+1) \nu_\ell^\ell - 2\ell \nu_\ell^{2\ell-1} \Big)^{-1/2}.
\end{align}

The $z$-coordinates of the singularities of $y_\ell(z)$ occur at roots of the discriminant of $F_\ell$, regarded as a polynomial in~$y$.  It is not hard to see that this discriminant is a polynomial of degree $2\ell$ in~$z$, but the polynomial itself is complicated.  We used software to confirm that for $\ell \le 16$, (iii)~does hold, and so \eqref{D_l asymp} is valid.  However, starting at $\ell = 14$ the discriminant has pairs of complex roots of magnitude less than $\rho_\ell$, increasingly many as $\ell$ grows, increasingly close to zero.  As a side note, software also suggests that the discriminant is always irreducible, which would imply that $\rho_\ell$ and $\nu_\ell$ are both of degree $2\ell$ over $\Q$.

We mention that the $y$-coordinates of the singularities of $y_\ell(z)$ occur at roots of the resultant of $F_\ell$ and $\partial_y F_\ell$, regarded as polynomials in~$z$, which is simply $R_\ell(y)$.  One can prove that $\nu_\ell$ is the unique root of $R_\ell$ of magnitude~$\le \nu_\ell$ indirectly, by observing that the Maclaurin series of $z_\ell'$ has this property.
\end{pb}

Finally, let us formalize the obvious statement that $D[\ell](z)$ goes to the Catalan generating function $C(z)$ as $\ell$ goes to~$\infty$.  Recall that the shifted Catalan function $y_C(z) := zC(z)$ has inverse $z_C(y) := y - y^2$, radius of convergence $1/4$, and limiting value $1/2$ as $z \to (1/4)^-$.  

\begin{prop} \label{D[infty]}
In the limit as $\ell \to \infty$,
\begin{enumerate}
\item[(i)]
$\rho_\ell$ increases monotonically from $\rho_1 = 3 - 2\sqrt{2}$ to $1/4$.
\smallbreak \item[(ii)]
$\nu_\ell$ increases monotonically from $\nu_1 = 1 - 1/\sqrt{2}$ to $1/2$.
\smallbreak \item[(iii)]
$D[\ell](z)$ decreases monotonically to $C(z)$ for all $z \in (0, 1/4)$.
\end{enumerate}
\end{prop}

\begin{proof}
The series for $z_\ell'$ gives~(ii), and then the series for $z_\ell$ gives~(i).  For~(iii), recall \eqref{coefficient comparison} and note that $D[\ell]_n = C_n$ for $\ell \ge n$.
\end{proof}

\begin{rem}
It is an amusing exercise to graph $z_\ell$ on $\R$.  Consider
\begin{equation*}
   z_\ell''(y) = \frac{(\ell - 1) (\ell - 2) y^{2\ell} + (\ell - 1) (\ell + 4) y^\ell + 2}
   {(y^\ell - 1)^3}.   
\end{equation*}
For $\ell \ge 3$, regard the numerator as a quadratic in $y^\ell$.  As such, its discriminant is $\ell^2 (\ell - 1) (\ell + 7)$ and both roots are negative.  It follows that for $\ell = 1$ or $\ell$ even, $z_\ell$ has no real inflection points, while for $\ell \ge 3$ and odd, it has two, both at negative $y$-values, which approach $-1$ from opposite sides as $\ell \to \infty$.

In the former case, $z_\ell$ has exactly two real real extrema: a maximum at $y = \nu_\ell$ and a minimum at some $y > 1$.  In the latter case it has these two extrema and possibly two more, a minimum and a maximum, both at negative $y$-values.  Observing that for $\ell$ odd, $z_\ell'(-1) = 8 - \ell$, and plotting the cases $\ell = 3, 5, 7$ with a computer, we find that these two additional extrema obtain for $\ell \ge 9$ and odd, and that they approach $y = -1$ from opposite sides as $\ell \to \infty$.  This is explained by the fact that for $\ell$ large, $z_\ell$ approximates $y$ for $|y| > 1$ and $y - y^2$ for $|y| < 1$.
\end{rem}

\section{Operations on dissections} \label{MaxOpenSec}

At this point we have proven all stated results except for Theorem~\ref{DiffQThm}, the formula for the 3-periodic quiddity BGF $Q(z, w)$ in terms of $P(z, w)$.  We now take the first step in the proof of this theorem: we construct a canonical representative of each equivalence class of 3-periodic dissections with the same quiddity.

Before we begin, we wish to emphasize two points concerning quiddity generating functions.  To put our results in context it would be natural to ask about the generating functions of the quiddities of the $\ell$-periodic dissections for arbitrary $\ell$.  Denote these functions by $Q[\ell]$, so that $Q = Q[3]$.

\begin{itemize}
\item
The UGF $Q[\ell](z)$ and the BGF $Q[\ell](z, w)$ are well-defined, but the MGF $Q[\ell](\vecw)$ is not.  This is because by Lemma~\ref{T(m)}, the quiddity determines the numbers of vertices and of cells in the dissection, but it does not determine the number of $(r+2)$-cells for each~$r$.

\smallbreak \item
Our construction of a canonical dissection associated to each quiddity works only in the case $\ell = 3$.  As far as we can see, it does not adapt to give a canonical $\ell$-periodic dissection associated to each $\ell$-periodic quiddity in general.  Consequently, we do not know anything about $Q[\ell](z, w)$ for $\ell \not= 3$.  To our knowledge, this question is new; hence the problem we formulated in Section~\ref{RGP}.
\end{itemize}

\subsection{Surgery} \label{QPS}
Here we define an operation on dissections which preserves the quiddity.  Consider a dissection of a convex $(n+2)$-gon.  Let us begin by collecting terminology from earlier sections:

\begin{itemize}
\item
The \textit{vertices} are labelled $0, 1, 2, \ldots, n+1$, in cyclic order.
\smallbreak \item
The \textit{edges} are the segments $(i, i+1)$ bounding the $(n+2)$-gon, $i \in \Z_{n+2}$.
\smallbreak \item
The \textit{base edge} is $(n+1, 0)$.
\smallbreak \item
The \textit{chords} are the non-crossing diagonals $(i, j)$ that make up the dissection.  Note that here $i$ and~$j$ are not cyclically adjacent.
\smallbreak \item
The \textit{cells} are the sub-polygons into which the $(n + 2)$-gon is dissected.
\smallbreak \item
The \textit{base cell} is the cell containing the base edge.
\smallbreak \item
The \textit{sides} of a cell are the edges and chords bounding it.
\end{itemize}

\medbreak
The next three definitions establish the concept of \textit{level,} which is a measure of a type of distance from any given cell to the base cell.

\begin{defn}
Consider a non-base cell~$\cC$.  Its \textit{base side} is the unique side with the following property: it is a chord, and of the two pieces into which it divides the $(n+2)$-gon, one contains the base cell and the other contains~$\cC$.
\end{defn}

\begin{defn}
Consider a cell.
\begin{itemize}
\item
Its \textit{parent} is the unique cell with which it shares its base side.
\smallbreak \item
Its \textit{children} are the cells of which it is the parent.
\smallbreak \item
Define its \textit{ancestors} and \textit{descendants} accordingly.
\end{itemize}
\end{defn}

In order to visualize the descendants of a cell, say $\cC$, divide the $(n+2)$-gon into two pieces along the base side of $\cC$.  The descendants of $\cC$ are precisely all cells in the same piece as $\cC$.

\begin{defn}
The \textit{level} of a cell is the number of ancestors it has.
\end{defn}

Some observations are in order.  The base cell has no parent and so is of level~0.  A cell of level~$L$ has exactly one ancestor of each level $0, 1, \ldots, L - 1$.  Its parent is of level $L-1$ and its children are of level $L+1$.  

The level and ancestors of a cell may be understood as follows. Consider a path from the cell to the base cell which stays in the interior of the $(n+2)$-gon and crosses the minimum number of chords.  The cell's level is the number of chords the path crosses, and its ancestors are the cells the path enters.

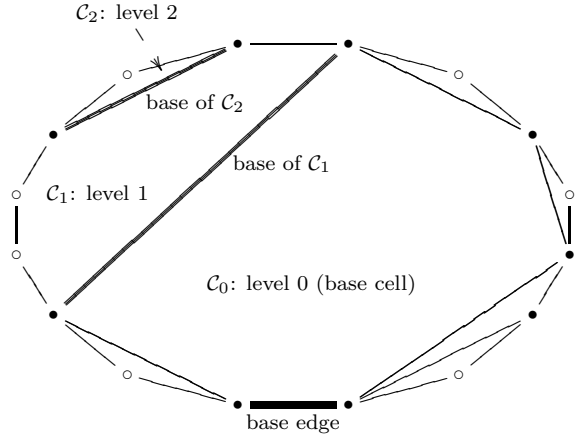
\begin{figure}[htb]
\footnotesize
$$
\xymatrix @!0 @R=0.40cm @C=0.49cm
{
&&&
\mbox{\rm $\cC_2$: level~2} \ar@{-->}[rdd]
&&&&&
\\
&&&&&&
\bullet \ar@{-}[llld]
\ar@{-}[lllllddd]
\ar@<0.2mm>@{-}[lllllddd]
\ar@<-0.2mm>@{-}[lllllddd]
\ar@{-}[rrr]
&&&
\bullet\ar@{-}[rrrd]
\ar@{-}[llllllllddddddddd]
\ar@<0.2mm>@{-}[llllllllddddddddd]
\ar@<-0.2mm>@{-}[llllllllddddddddd]
\\
&&&
\circ \ar@{-}[lldd]
&&&&&&&&&
\circ \ar@{-}[rrdd]
\\
&&&&& \!\!\! \mbox{\rm base of $\cC_2$}
\\
& \bullet \ar@{-}[ldd]
&&&&&&&&&&&&&
\bullet \ar@{-}[rdd]
\ar@{-}[llllluuu]
\ar@{-}[llllluuu]
\ar@{-}[rdddd]
\ar@{-}[rdddd]
\\
&&&&&&& \mbox{\rm\ \ base of $\cC_1$}
\\
\circ \ar@{-}[dd]
&& \mbox{\rm\ \ $\cC_1$: level~1}
&&&&&&&&&&&&&
\circ \ar@{-}[dd]
\\
\\
\circ \ar@{-}[rdd]
&&&&&&&&
&&&&&&&
\bullet \ar@{-}[ldd]
\ar@{-}[llllllddddd]
\ar@{-}[llllllddddd]
\\
&&&&&&&&
\mbox{\rm{$\cC_0$: level~0 (base cell)}}
\\
& \bullet \ar@{-}[rrdd]
\ar@{-}[rrrrrddd]
\ar@{-}[rrrrrddd]
&&&&&&&&&&&&&
\bullet \ar@{-}[lldd]
\\
\\
&&&
\circ \ar@{-}[rrrd]
&&&&&&&&&
\circ \ar@{-}[llld]
\\
&&&&&&
\bullet \ar@{-}[rrr]
\ar@<0.2mm>@{-}[rrr]
\ar@<-0.2mm>@{-}[rrr]
\ar@<0.4mm>@{-}[rrr]
\ar@<-0.4mm>@{-}[rrr]_{\mbox{\rm base edge}}
&&&
\bullet \ar@{-}[rrrrruuu]
}
$$
\caption*{A level~2 cell $\cC_2$ and its two ancestors: its parent $\cC_1$ and its grandparent $\cC_0$, the base cell (hollow dots represent sub-dissections).}
\normalsize
\end{figure}

The next two definitions introduce surgery.  Consider a dissection of an $(n+2)$-gon containing an $(r+2)$-cell with vertices $v_0, \ldots, v_{r+1}$, where
\begin{equation} \label{cell vertices}
   0 \le v_0 < v_1 < \cdots < v_r < v_{r+1} \le n + 1.
\end{equation}
The sides of the cell are the segments $(v_s, v_{s+1})$.  At $s = r+1$ we take this to mean $(v_{r+1}, v_0)$.  It is important to note that this is the base side.

\begin{defn}
Two sides of a cell are \textit{distant} if ``the cell has vertices properly between them''.  Thus sides $(v_s, v_{s+1})$ and $(v_{s'}, v_{s'+1})$ with $s < s'$ are distant if 
\begin{equation*}
   s' \ge s+3, \qquad s \ge (s'+3) - (r+2).
\end{equation*}
\end{defn}

\begin{defn}
\textit{Surgery} may be performed on any two distant sides of a cell which are both chords of the dissection; neither is an edge.  It ``replaces them by the other two sides of the quadrilateral formed by their vertices''.  Thus if $(v_s, v_{s+1})$ and $(v_{s'}, v_{s'+1})$ are distant sides and also chords, surgery removes them from the dissection and replaces them by the chords $(v_{s+1}, v_{s'})$ and $(v_{s'+1}, v_s)$.
\end{defn}

\begin{figure}[htbp]
\footnotesize
$$
\xymatrix @!0 @R=0.35cm @C=0.6cm
{
&\circ \ar@{-}[ldd] \ar@{-}[rr] && v_s \ar@{-}[rdd]
\\
\\
v_{s+1} \ar@{-}[dd] \ar@{-}[rrruu] &&&& \circ \ar@{-}[dd]
\\
\\
\circ \ar@{-}[rdd] &&&& v_{s'+1} \ar@{-}[ldd] \ar@{-}[llldd]
\\
\\
& v_{s'} \ar@{-}[rr] && \circ
}
\mkern54mu
\xymatrix @!0 @R=0.35cm @C=0.6cm
{
\\
\\
\\
\longmapsto
}
\mkern54mu
\xymatrix @!0 @R=0.35cm @C=0.6cm
{
& \circ \ar@{-}[ldd] \ar@{-}[rr] && v_s \ar@{-}[rdd]
\\
\\
v_{s+1} \ar@{-}[dd] \ar@{-}[rdddd] &&&& \circ \ar@{-}[dd]
\\
\\
\circ \ar@{-}[rdd] &&&& v_{s'+1} \ar@{-}[ldd] \ar@{-}[luuuu]
\\
\\
&v_{s'} \ar@{-}[rr] && \circ
}
$$
\caption*{Surgery on a cell (again, hollow dots represent sub-dissections)}
\normalsize
\end{figure}
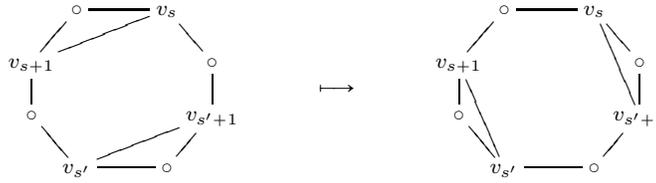

The result of a surgery is a new dissection, as the modified set of chords is still non-crossing.  Surgery is reversible: the two new chords it creates are themselves distant sides of a newly created cell, and surgery on them inverts the original surgery.  Surgery alters exactly three cells in the dissection: the cell whose distant sides are replaced, and the two cells with which it shared those sides.  The main point is the following lemma, which is obvious.

\begin{lem} \label{ObvLem}
Surgery does not change the quiddity of the dissection.
\end{lem}

This suggests a question: is the converse true?  If two dissections have the same quiddity, can one be transformed into the other by a sequence of surgeries?  In general, we do not know.  The main result of this section is that in the 3-periodic case, the answer is yes.

\begin{defn}
Consider two distant sides of a cell which are both chords.  If one of them is the base side, surgery on them is \textit{opening.}  Otherwise it is \textit{closing.}
\end{defn}

Opening surgery is never possible on the base cell, as its base side is an edge.  Opening and closing surgeries are mutually inversive.

Let us describe the effect of an opening surgery.  Suppose that $\cC_L$ is a non-base cell of level~$L$ with vertices~\eqref{cell vertices}.  Write $\cC_{L-1}$ for its parent cell, with which it shares its base side $(v_{r+1}, v_0)$.  Assume that $\cC_L$ has a non-base side $(v_s, v_{s+1})$ which is distant from the base side, i.e., $2 \le s \le r-2$, and is a chord.  Write $\cC_{L+1}$ for the child cell of $\cC_L$ with which it shares $(v_s, v_{s+1})$.

In this setting $(v_s, v_{s+1})$ and $(v_{r+1}, v_0)$ are eligible for opening surgery, which merges the parent $\cC_{L-1}$ and the child $\cC_{L+1}$ into a single cell $\cC'_{L-1}$ of level $L-1$, and divides the original cell $\cC_L$ into two cells $\cC'_L$ and $\cC''_L$, both of level $L$.  The two new level $L$ cells are both children of the new level $L-1$ cell, and their base sides are the two new chords created by the surgery.

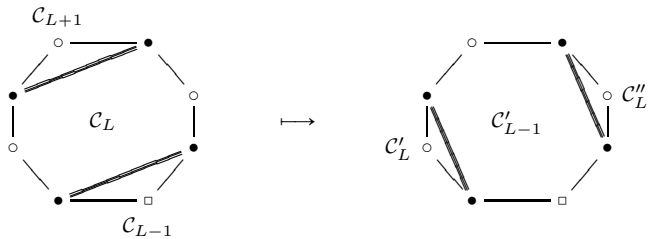
\begin{figure}[htbp]
\footnotesize
$$
\xymatrix @!0 @R=0.35cm @C=0.6cm
{
& \cC_{L+1} &&& \\
& \circ \ar@{-}[ldd] \ar@{-}[rr] && \bullet \ar@{-}[rdd]\\
\\
\bullet \ar@{-}[dd]
\ar@{-}[rrruu] \ar@<0.2mm>@{-}[rrruu] \ar@<-0.2mm>@{-}[rrruu]
&&&& \circ \ar@{-}[dd]\\
&& \cC_L &&
\\
\circ \ar@{-}[rdd] &&&& \bullet \ar@{-}[ldd]
\ar@{-}[llldd] \ar@<0.2mm>@{-}[llldd] \ar@<-0.2mm>@{-}[llldd]\\
\\
& \bullet \ar@{-}[rr] && \square \\
&&& \cC_{L-1} &
}
\mkern54mu
\xymatrix @!0 @R=0.35cm @C=0.6cm
{
\\
\\
\\
\\
\longmapsto
}
\mkern18mu
\xymatrix @!0 @R=0.35cm @C=0.6cm
{
\\
&& \circ \ar@{-}[ldd] \ar@{-}[rr] && \bullet \ar@{-}[rdd]
\\
\\
& \bullet \ar@{-}[dd]
\ar@{-}[rdddd] \ar@<0.2mm>@{-}[rdddd] \ar@<-0.2mm>@{-}[rdddd] 
&&&& \circ \ar@{-}[dd] &
\mkern-27mu \cC''_L
\\
&&& \cC'_{L-1} &&
\\
\mkern27mu \cC'_L
& \circ \ar@{-}[rdd] &&&& \bullet \ar@{-}[ldd]
\ar@{-}[luuuu] \ar@<0.2mm>@{-}[luuuu] \ar@<-0.2mm>@{-}[luuuu]
\\
\\
&& \bullet \ar@{-}[rr] && \square
\\
}
$$
\caption{Opening surgery on a cell $\cC_L$: the base sides of $\cC_L$, $\cC_{L+1}$, $\cC'_L$, and $\cC''_L$ are emphasized, hollow dots represent sub-dissections, and the base is contained in the square.}
\label{OS}
\normalsize
\end{figure}

All other cells in the dissection remain unchanged.  The levels of the descendants of $\cC_{L+1}$ all decrease by~$2$, and no other levels change.  Heuristically, we think of opening surgeries as ``bringing cells closer to the base cell'', by ``opening ancestor cells outward towards descendant cells''; hence the terminology.

\subsection{3-periodic surgery} \label{3dS}
Surgery does not preserve 3-periodic dissections.  For example, there is a 3-periodic dissection of the $11$-gon into a base nonagon and two level~1 triangles which after surgery has a base pentagon, a level~1 hexagon, and a level~2 quadrilateral: take the original chords to be $(2, 4)$ and $(7, 9)$ in $(0, 1, \ldots, 10)$.  However, there is a natural type of surgery which does preserve 3-periodic dissections.

\begin{defn}
Given a 3-periodic dissection, we assign an element of $\Z_3$ to each edge, chord, and cell, its \textit{$\Z_3$-index}.  The procedure is recursive on level.  To begin, assign the index~0 to the base edge of the base cell.  Once the sides of all cells of level~$< L$ are indexed, the base sides of all cells of level~$L$ will have been indexed.  To index their remaining sides, increase the indices in increments of~1 going counterclockwise around each cell.  The index of a cell is the index of its base side.
\end{defn}

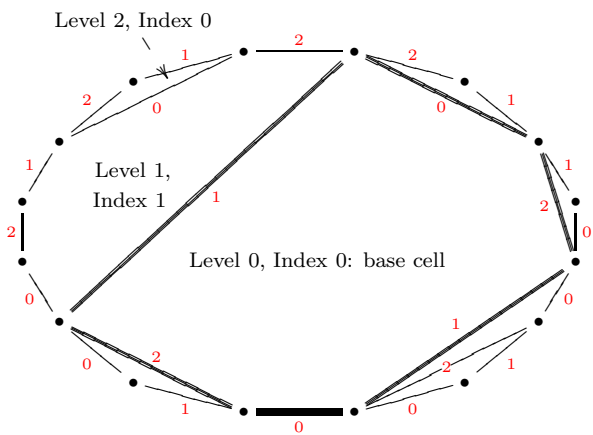
\begin{figure}[htb]
\footnotesize
$$
\xymatrix @!0 @R=0.40cm @C=0.49cm
{
&&&
\mbox{\rm{Level~2, Index~0}} \ar@{-->}[rdd]
&&&&&
\\
&&&&&&
\bullet \ar@{-}[llld]_{\textcolor{red}{1}}
\ar@{-}[lllllddd]^{\textcolor{red}{0}}
\ar@{-}[rrr]^{\textcolor{red}{2}}
&&&
\bullet\ar@{-}[rrrd]^{\textcolor{red}{2}}
\ar@{-}[llllllllddddddddd]
\ar@<0.2mm>@{-}[llllllllddddddddd]
\ar@<-0.2mm>@{-}[llllllllddddddddd]^{\textcolor{red}{1}}
\\
&&&
\bullet \ar@{-}[lldd]_{\textcolor{red}{2}}
&&&&&&&&&
\bullet \ar@{-}[rrdd]^{\textcolor{red}{1}}
\\
\\
& \bullet \ar@{-}[ldd]_{\textcolor{red}{1}}
&&&&&&&&&&&&&
\bullet \ar@{-}[rdd]^{\textcolor{red}{1}}
\ar@{-}[llllluuu]
\ar@<0.2mm>@{-}[llllluuu]
\ar@<-0.2mm>@{-}[llllluuu]^{\textcolor{red}{0}}
\ar@{-}[rdddd]
\ar@<0.2mm>@{-}[rdddd]
\ar@<-0.2mm>@{-}[rdddd]_{\textcolor{red}{2}}
\\
&&&
\mbox{\rm{Level~1,}}
\\
\bullet \ar@{-}[dd]_{\textcolor{red}{2}}
&&&
\mbox{\rm{Index~1\ }}
&&&&&&&&&&&&
\bullet \ar@{-}[dd]^{\textcolor{red}{0}}
\\
\\
\bullet \ar@{-}[rdd]_{\textcolor{red}{0}}
&&&&&&&&
\mbox{\rm{Level~0, Index~0: base cell}}
&&&&&&&
\bullet \ar@{-}[ldd]^{\textcolor{red}{0}}
\ar@{-}[llllllddddd]
\ar@<0.2mm>@{-}[llllllddddd]
\ar@<-0.2mm>@{-}[llllllddddd]_{\textcolor{red}{1}}
\\
\\
& \bullet \ar@{-}[rrdd]_{\textcolor{red}{0}}
\ar@{-}[rrrrrddd]
\ar@<0.2mm>@{-}[rrrrrddd]
\ar@<-0.2mm>@{-}[rrrrrddd]^{\textcolor{red}{2}}
&&&&&&&&&&&&&
\bullet \ar@{-}[lldd]^{\textcolor{red}{1}}
\\
\\
&&&
\bullet \ar@{-}[rrrd]_{\textcolor{red}{1}}
&&&&&&&&&
\bullet\ar@{-}[llld]^{\textcolor{red}{0}}
\\
&&&&&&
\bullet \ar@{-}[rrr]
\ar@<0.2mm>@{-}[rrr]
\ar@<-0.2mm>@{-}[rrr]
\ar@<0.4mm>@{-}[rrr]
\ar@<-0.4mm>@{-}[rrr]_{\textcolor{red}{0}}
&&&
\bullet \ar@{-}[rrrrruuu]|-{\textcolor{red}{2}}
}
$$
\caption{The $\Z_3$-indices of a 3-periodic dissection of a 16-gon}
\label{MO not BO}
\normalsize
\end{figure}

\begin{lem} \label{3d surgery}
Consider two sides of a cell in a 3-periodic dissection.
\begin{enumerate}
\item[(i)]
If the sides have the same $\Z_3$-index, then they are distant.
\smallbreak \item[(ii)]
If the sides are distant and are chords, then the dissection produced by surgery on them is 3-periodic if and only if they have the same $\Z_3$-index.
\end{enumerate}
\end{lem}

\begin{defn}
\textit{3-periodic surgery} on a 3-periodic dissection is surgery on two sides of a cell which are chords of the same $\Z_3$-index.
\end{defn}

\begin{lem}
3-periodic surgery does not alter the $\Z_3$-indices:
\begin{enumerate}
\item[(i)]
The two new sides have the same $\Z_3$-index as the sides they replace.
\smallbreak \item[(ii)]
The $\Z_3$-indices of all other sides remain the same.
\end{enumerate}
\end{lem}

The proofs of these lemmas are trivial and are left to the reader.  Note that the inverse of a 3-periodic surgery is a 3-periodic surgery.

The remainder of this section is devoted to proving that any two 3-periodic dissections with the same quiddity are linked by a series of 3-periodic surgeries.  The strategy is to show that each quiddity equivalence class of 3-periodic dissections contains a unique dissection on which no 3-periodic opening surgeries are possible.

\begin{defn}
Consider a 3-periodic dissection and a cell within it.
\begin{enumerate}
\item[(i)]
The cell is \textit{maximally open} if it admits no 3-periodic opening surgeries.
\smallbreak \item[(ii)]
The dissection is \textit{maximally open} if all of its cells are maximally open.
\end{enumerate}
\end{defn}

\begin{lem} \label{3dS count}
Consider a cell in a 3-periodic dissection.  The number of 3-periodic opening surgeries it admits is equal to the number of non-base sides it has which are chords of the same $\Z_3$-index as its base side.
\end{lem}

\begin{cor} \label{MO}
A 3-periodic dissection is maximally open if and only if in every non-base cell, every non-base side of the same $\Z_3$-index as the base side is an edge.
\end{cor}

These two results are also obvious.  In order to state the next result, suppose that $\cD$ is a 3-periodic dissection in which all cells of level~$> L$ are maximally open, but not all cells of level~$L$ are maximally open.  Let $\sigma_L$ be the total number of 3-periodic opening surgeries on cells of level~$L$ admitted by the entire dissection.

Fix a cell $\cC_L$ in $\cD$ of level~$L$ which admits at least one 3-periodic opening surgery.  Fix such a surgery, and let $\cD'$ be the new 3-periodic dissection it produces.  Let $\sigma'_L$ be the total number of 3-periodic opening surgeries on cells of level~$L$ admitted by $\cD'$.

\begin{lem}
The 3-periodic dissection $\cD'$ has the following properties:
\begin{enumerate}
\item[(i)]
All cells of level~$> L$ are maximally open.
\smallbreak \item[(ii)]
$\sigma'_L = \sigma_L - 1$.
\end{enumerate}
\end{lem}

\begin{proof}
We use the notation of Figure~\ref{OS}; subscripts denote levels.  The surgery merges the parent $\cC_{L-1}$ of $\cC_L$ with one of its children, $\cC_{L+1}$, to form $\cC'_{L-1}$.  This divides $\cC_L$ into $\cC'_L$ and $\cC''_L$.  The other cells are not changed, and the only change in their levels is that those of the descendants of $\cC_{L+1}$ all decrease by~2.

For~(i), all cells of level~$> L$ in $\cD'$ were cells of level~$> L$ in $\cD$, and so they are maximally open.  For~(ii), there are three types of level~$L$ cells in $\cD'$:
\begin{itemize}
\item
those which were level~$L$ cells in $\cD$ other than $\cC_L$;
\item
those which were children of $\cC_{L+1}$ in $\cD$;
\item
the two new cells $\cC'_L$ and $\cC''_L$.
\end{itemize}

The number of opening surgeries in cells of the first type is not changed by the surgery, and cells of the second type are maximally open.  An application of Lemma~\ref{3dS count} shows that the sum of the number of opening surgeries in $\cC'_L$ and $\cC''_L$ is one less than the number of opening surgeries in $\cC_L$, because the base side of $\cC_{L+1}$ is no longer available.  The result follows.
\end{proof}

Let us remark that the new dissection $\cD'$ may admit arbitrarily many more opening surgeries than the original dissection $\cD$.  However, they will all be in $\cC'_{L-1}$, which is of level $L-1$.  Therefore an obvious induction argument starting from the highest level present in $\cD$ and proceeding downward yields the following result.

\begin{prop} \label{3d to MO}
Any 3-periodic dissection can be transformed into a maximally open 3-periodic dissection by a sequence of 3-periodic opening surgeries.
\end{prop}

Figure~\ref{2 openings} gives a simple example of three dissections of a hexadecagon.  The chords are labelled by their $\Z_3$-indices.  In the initial dissection on the left, there is one cell in each of the levels $0$, $1$, $2$, $3$, and~$4$.  The only possible opening surgery is on the level~$2$ cell, producing the middle dissection, which has one level~$0$ cell (as always), one level~$1$ cell, and three level~$2$ cells.  It too has only one opening surgery, on the level~$1$ cell, producing the dissection on the right.  It is maximally open and has one level~$0$ cell, two level~$1$ cells, and two level~$2$ cells.

\begin{figure}[htbp]
\scriptsize
$$
\xymatrix @!0 @R=0.35cm @C=0.30cm
 {
&&&&&&\bullet\ar@{-}[llld]\ar@{-}[rrr]\ar@{-}[rrrrrrd]|-{\textcolor{red}{1}}
\ar@{-}[rrrrrrrrddddddddd]_{\textcolor{red}{2}}&&&\bullet\ar@{-}[rrrd]
\\
&&&\bullet\ar@{-}[lldd]&&&&&&&&& \bullet\ar@{-}[rrdd]\\
\\
&\bullet\ar@{-}[ldd]&&&&&&&&&&&&&\bullet\ar@{-}[rdd]\\
\\
\bullet\ar@{-}[dd]&&&&&&&&&&&&&&&\bullet\ar@{-}[dd]\\
&&&&&&&&&&&&\\
\bullet\ar@{-}[rdd]&&&&&&&&&&&&&&&\bullet\ar@{-}[ldd]\\
\\
&\bullet\ar@{-}[rrdd]\ar@{-}[rrrrrrrrddd]^{\textcolor{red}{2}}&&&&&&&&&&&&&\bullet\ar@{-}[lldd]\\
\\
&&&\bullet\ar@{-}[rrrd]\ar@{-}[rrrrrrd]|-{\textcolor{red}{1}}&&&&&&&&& \bullet\ar@{-}[llld]\\
&&&&&&\bullet\ar@{-}[rrr]\ar@<0.2mm>@{-}[rrr]\ar@<-0.2mm>@{-}[rrr]&&&\bullet
}
\ \
\xymatrix @!0 @R=0.35cm @C=0.30cm
 {
&&&&&&\bullet\ar@{-}[llld]\ar@{-}[rrr]\ar@{-}[rrrrrrd]_{\textcolor{red}{1}}\ar@{-}[lllllddddddddd]^{\textcolor{red}{2}}&&&\bullet\ar@{-}[rrrd]
\\
&&&\bullet\ar@{-}[lldd]&&&&&&&&& \bullet\ar@{-}[rrdd]\\
\\
&\bullet\ar@{-}[ldd]&&&&&&&&&&&&&\bullet\ar@{-}[rdd]\\
\\
\bullet\ar@{-}[dd]&&&&&&&&&&&&&&&\bullet\ar@{-}[dd]\\
&&&&&&&&&&&&\\
\bullet\ar@{-}[rdd]&&&&&&&&&&&&&&&\bullet\ar@{-}[ldd]\\
\\
&\bullet\ar@{-}[rrdd]&&&&&&&&&&&&&\bullet\ar@{-}[lldd]\ar@{-}[lllllddd]_{\textcolor{red}{2}}\\
\\
&&&\bullet\ar@{-}[rrrd]\ar@{-}[rrrrrrd]^{\textcolor{red}{1}}&&&&&&&&& \bullet\ar@{-}[llld]\\
&&&&&&\bullet\ar@{-}[rrr]\ar@<0.2mm>@{-}[rrr]\ar@<-0.2mm>@{-}[rrr]&&&\bullet
}
\ \
\xymatrix @!0 @R=0.35cm @C=0.30cm
{
&&&&&&\bullet\ar@{-}[llld]\ar@{-}[rrr]\ar@{-}[lllllddddddddd]_{\textcolor{red}{2}}\ar@{-}[lllddddddddddd]^{\textcolor{red}{1}}&&&\bullet\ar@{-}[rrrd]
\\
&&&\bullet\ar@{-}[lldd]&&&&&&&&& \bullet\ar@{-}[rrdd]\ar@{-}[lllddddddddddd]_{\textcolor{red}{1}}\\
\\
&\bullet\ar@{-}[ldd]&&&&&&&&&&&&&\bullet\ar@{-}[rdd]\\
\\
\bullet\ar@{-}[dd]&&&&&&&&&&&&&&&\bullet\ar@{-}[dd]\\
&&&&&&&&&&&&\\
\bullet\ar@{-}[rdd]&&&&&&&&&&&&&&&\bullet\ar@{-}[ldd]\\
\\
&\bullet\ar@{-}[rrdd]&&&&&&&&&&&&&\bullet\ar@{-}[lldd]\ar@{-}[lllllddd]_{\textcolor{red}{2}}\\
\\
&&&\bullet\ar@{-}[rrrd]&&&&&&&&& \bullet\ar@{-}[llld]\\
&&&&&&\bullet\ar@{-}[rrr]\ar@<0.2mm>@{-}[rrr]\ar@<-0.2mm>@{-}[rrr]&&&\bullet
}
$$
\caption{Opening surgeries producing a maximally open dissection}
\label{2 openings}
\normalsize
\end{figure}
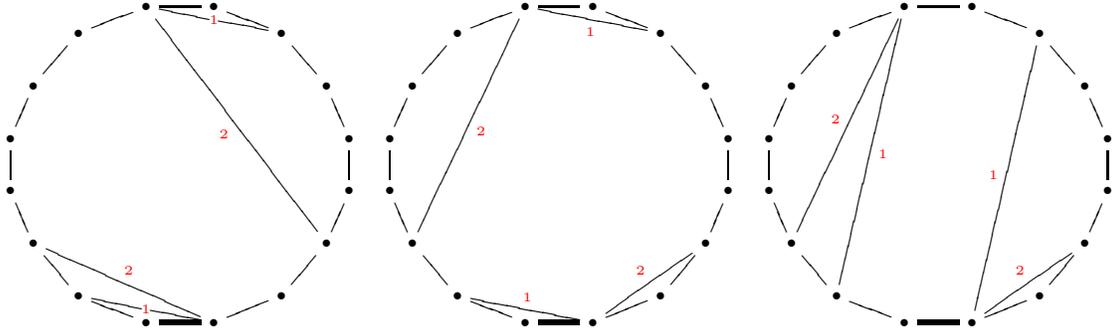

\begin{rem}
From a practical standpoint, a dissection containing cells admitting multiple opening surgeries may be transformed into a maximally open dissection more efficiently by performing what might be called ``3-periodic maximal opening surgery'' in place of 3-periodic opening surgery.  Given such a cell, consider the polygon formed by the vertices of its base side and all of its chord sides of the same $\Z_3$-index as its base side.  Half of the sides of this polygon are sides of the cell and half are not part of the dissection.  Delete the former from the dissection and replace them by the latter.  This process is equivalent to a sequence of ordinary 3-periodic opening surgeries.
\end{rem}

\subsection{Operations enlarging the polygon} \label{BUBA}
In this section we recall from \cite{Val} two operations on 3-periodic dissections.  They are quite different in nature from surgery, in that both change not only the quiddity but even the number of vertices of the polygon.  Consider an arbitrary 3-periodic dissection of an $(n+2)$-gon, with quiddity
\begin{equation*}
(a_0, \ldots, a_i, a_{i+1}, \ldots, a_{n+1}).
\end{equation*}

\begin{itemize}

\item
\textit{Blow-up:}
Add a triangle to the dissection, attaching it to the $(i, i+1)$-edge.
This results in a 3-periodic dissection of an $(n+3)$-gon, with quiddity
\begin{equation*}
(a_0,\ldots, a_{i-1}, a_i+1,\, 1,\, a_{i+1}+1, a_{i+2}, \ldots, a_{n+1}).
\end{equation*}

\item
\textit{Expansion:}
This operation expands one of the
sub-polygons contacting the $i^{\thup}$ vertex,
replacing that vertex with three new edges.
The construction is as follows:
expand the $i^{\thup}$ vertex to an edge
and call its endpoints $i'$ and $i''$.
Place two new vertices along this new edge,
dividing it into three edges.
Of the $a_i - 1$ chords contacting
the $i^{\thup}$ vertex in the original dissection,
allot $a'_i - 1$ of them to $i'$ and $a''_i - 1$ of them to $i''$,
where $a'_i$ and $a''_i$ are positive integers summing to $a_i+1$.

\begin{figure}[!htb]
\footnotesize
$$
 \xymatrix @!0 @R=0.32cm @C=0.45cm
{
&&&\ar@{-}[dddddddd]\ar@{-}[rrd]&
\\
&\ar@{-}[rrddddddd]&&&& \\
\\
\ar@{-}[rrrddddd]&&&&&&&\ar@{-}[llllddddd]\\
\\
\ar@{-}[rrrddd]&&&&&&&&\ar@{-}[lllllddd]\\
\\
&\ar@{-}[rrd]&&&&&&& \ar@{-}[llllld]\\
&&& i \ar@{-}[rruuuuuuu]&
}
\qquad\qquad\qquad
\xymatrix @!0 @R=0.32cm @C=0.45cm
{
&&&&\ar@{-}[ldddddddd]\ar@{-}[rrd]&
\\
&\ar@{-}[rrddddddd]&&&&& \\
\\
\ar@{-}[rrrddddd]&&&&&&&&\ar@{-}[lddddd]\\
\\
\ar@{-}[rrrddd]&&&&&&&&&\ar@{-}[llddd]\\
\\
&\ar@{-}[rrd]&&&&&&&& \ar@{-}[lld]\\
&&& i' \ar@{-}[rd]
&&&& i'' \ar@{-}[luuuuuuu]\\
&&&& \cdot \ar@{-}[r]& \cdot \ar@{-}[rru]
}
$$
\caption*{Before and after an expansion}
\normalsize
\end{figure}
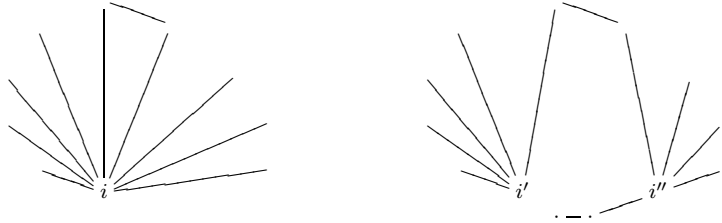

This yields a 3-periodic dissection of an $(n+5)$-gon,
with no chords contacting the two vertices between $i'$ and $i''$.
Its quiddity is
\begin{equation*}
(a_0, \ldots, a_{i-1}, a'_i,\, 1,\, \,1,\, a''_i, a_{i+1}, \ldots, a_{n+1}).
\end{equation*}

\end{itemize}

Because the vertices are cyclically ordered,
both operations are defined for $0 \le i \le n+1$.
The blow-up operation is well known and was the main technical tool
used by Conway and Coxeter in~\cite{CoCo}.
To our knowledge, expansion was first considered in~\cite{Val},
where the two operations are used together to prove Theorem~\ref{ValRMSthm}.
The idea of the proof is to regard the operations
as acting on quiddities rather than dissections.
Let us briefly explain it.
The following lemma is an easy exercise.

\begin{lem}[\cite{Val}, Section~1.2]
\label{operations}
Blow-up and expansion preserve solutions of~\eqref{SLMatEqInt}.
Blow-up increases $N$ by~$1$, expansion increases $N$ by~$3$,
and both increase $T$ by~$3$.
\end{lem}

Clearly any 3-periodic dissection may be obtained by applying a sequence of blow-ups and expansions to a triangle.  The next theorem is the analog of this observation for solutions of~\eqref{SLMatEqInt}.  Combining it with Lemma~\ref{operations} leads to Theorem~\ref{ValRMSthm}.

\begin{thm}[\cite{Val}, Theorem~2]
\label{ValRMSthm2}
Any positive solution $(a_1, \ldots, a_N)$ of~\eqref{SLMatEqInt}
may be obtained from the quiddity $(1, 1, 1)$ of the triangle
via a sequence of blow-ups and expansions.
\end{thm}

The terminology ``blow-up'' comes from the theory of toric varieties; see~\cite{Ful}, pp.~43--44.  Expansion also has a geometric interpretation: it adds a half-turn to the corresponding rational fan.  We discuss connections between~\eqref{SLMatEqInt} and toric surfaces in Section~\ref{ProjPlane}.

\subsection{Quiddity classes of dissections}
We have come to our main technical result:

\begin{thm}
\label{MainTechThm}
Every equivalence class of 3-periodic dissections with the same quiddity contains a unique maximally open 3-periodic dissection.
\end{thm}

\begin{proof}
Keeping in mind Lemma~\ref{ObvLem}, we see by Proposition~\ref{3d to MO} that every quiddity class of 3-periodic dissections contains at least one maximally open dissection.  Therefore the following proposition proves the theorem.
\end{proof}

\begin{prop}
If two maximally open 3-periodic dissections have the same quiddity, then they are equal.
\end{prop}

\begin{rem}
It would be natural to try to induct by applying the result to an appropriate sub-dissection, for example, one obtained by removing a single cell of maximal level.  But this does not work, because the quiddity alone does not reveal the cell structure.

Instead we must devise a process which transforms a maximally open 3-periodic dissection into a smaller one, using as input information only the quiddity.  The key is to construct what may be thought of as a quotient dissection.  The blow-up and expansion operations from Section~\ref{BUBA} play a crucial role.  
\end{rem}

\begin{proof}
Let $\cD$ and $\tilde\cD$ be two maximally open 3-periodic dissections of an $(n+2)$-gon which have the same quiddity, and assume the proposition for smaller~$n$.  Label the vertices $0, 1, \ldots, n+1$ and write $(a_0, a_1, \ldots, a_{n+1})$ for the quiddity.  If the quiddity values are all~$1$, both dissections are trivial, so we may dispense with this case.

Refer to cells with no children as \textit{terminal.}  In a non-trivial dissection, a terminal cell is a non-base cell with only one chord side, $(i, j)$.  Here $j-i \ge 2$ and either $0 < i$, or $j < n+1$, or both.  The vertices of the cell are the consecutive integers $i, i+1, \ldots, j$, and their quiddity values satisfy
\begin{equation} \label{QVexp}
   a_i > 1, \quad a_{i+1} = a_{i+2} = \cdots = a_{j-1} = 1, \quad a_{j} > 1.
\end{equation}
In the 3-periodic case, $3$~divides $j-i+1$.

There is always at least one terminal cell; for example, any cell of maximal level.  Let $i, \ldots, j$ be the vertices of a terminal cell of $\tilde\cD$.  Name this cell $\tilde\cC$.

In the dissection $\cD$, we cannot conclude a priori that $i, \ldots, j$ are the vertices of a terminal cell; we only know that their quiddity values are as in~\eqref{QVexp}.  These quiddity values do permit us to conclude that $i, \ldots, j$ are all contained in a single cell of $\cD$.  Name this cell $\cC$.  For illustrative examples, see Figure~\ref{2 openings}.

\medbreak \noindent
\textit{The case $j > i + 2$.}
Here $\cC$ is a $3d$-gon for some $d \ge 2$, and so $j \ge i + 5$.  Because no chords emanate from the vertices $i+1$, $i+2$, and $i+3$, we may create from $\cD$ a new dissection $\cD'$ by erasing them and connecting $i$ to $i+4$ by a single edge.  This is the inverse of the expansion operation: $\cD'$ is a 3-periodic dissection of an $(n-1)$-gon with vertices $0, \ldots, i, i+4, \ldots, n+1$.

The sole difference between $\cD'$ and $\cD$ is that in $\cD'$, the cell $\cC$ is replaced by a cell $\cC'$ obtained by deleting the vertices $i+1$, $i+2$, and $i+3$.  All other cells remain the same, all chords are the same, all base sides are the same, and the levels, $\Z_3$-indices, and quiddity values of all components of $\cD'$ are the same as they were in $\cD$.  In particular, $\cD'$ is maximally open and has quiddity
\begin{equation*}
   (a_0, \ldots, a_i, a_{i+4}, \ldots, a_{n+1}).
\end{equation*}

The same inverse expansion operation may be performed on $\tilde\cD$, producing a 3-periodic dissection $\tilde\cD'$ of an $(n-1)$-gon which differs from $\tilde\cD$ only in that the vertices $i+1$, $i+2$, and $i+3$ are deleted from the terminal cell $\tilde\cC$ to create a new cell $\tilde\cC'$.  As before, everything else remains the same, so $\tilde\cD'$ is also maximally open, and it has the same quiddity as $\cD'$.  Therefore, by the induction hypothesis, $\cD'$ and $\tilde\cD'$ are equal.  Performing expansion surgery on the cells $\cC'$ and $\tilde\cC'$ returns the smaller dissections $\cD'$ and $\tilde\cD'$ to the originals $\cD$ and $\tilde\cD$, so they too are equal.

\medbreak \noindent
\textit{The case $j = i + 2$.}
Here the terminal cell $\tilde\cC$ is a triangle with vertices $i$, $i+1$, and $i+2$.  Let $\tilde\cD'$ be the dissection obtained by removing this cell from $\tilde\cD$.  This is the inverse of the blow-up operation: $\tilde\cD'$ is a 3-periodic dissection of an $(n+1)$-gon with vertices $0, \ldots, i, i+2, \ldots, n+1$.

Clearly the cells, base sides, levels, and $\Z_3$-indices in $\tilde\cD'$ are the same as they were in $\tilde\cD$.  It follows that $\tilde\cD'$ is maximally open and has quiddity
\begin{equation} \label{QVblup}
   (a_0, \ldots, a_{i-1}, a_i - 1, a_{i+2} - 1, a_{i+3}, \ldots, a_{n+1}).
\end{equation}

Now consider the cell $\cC$ in the dissection $\cD$ containing the vertices $i$, $i+1$, and $i+2$.  Suppose that it contains no other vertices.  Then it is a terminal triangle, and we may perform the same inverse blow-up operation, removing $\cC$ from $\cD$ to produce a 3-periodic dissection $\cD'$ which is maximally open and has quiddity~\eqref{QVblup}.  By the induction hypothesis, $\cD' = \tilde\cD'$, and hence $\cD = \tilde\cD$.

Thus the proof will be complete if we derive a contradiction from the assumption that $\cC$ is not a triangle.  Suppose it is an $(r+2)$-cell with vertices $v_0 < \ldots < v_{r+1}$, where $r+2 = 3d$ for some $d \ge 2$.  For some $s \in \{0, 1, \ldots, r-1\}$, we have
\begin{equation*}
   v_s = i, \qquad v_{s+1} = i+1, \qquad v_{s+2} = i+2.
\end{equation*}

Because the quiddity values $a_i$ and $a_{i+2}$ exceed~$1$, the sides $(v_{s-1}, v_s)$ and $(v_{s+2}, v_{s+3})$ of $\cC$ are chords.  Moreover, they have the same $\Z_3$-index, so we may perform 3-periodic surgery on them.

Because $\cD$ is maximally open, this surgery must be a closing surgery: neither of the two chords it removes can be the base side $(v_{r+1}, v_0)$ of $\cC$.  Thus in fact $s \in \{1, \ldots, r-2\}$.

Let $\hat\cD$ be the dissection produced by this closing surgery.  It has the same quiddity as $\cD$ and $\tilde\cD$, and it contains the triangle $(i,\, i+1,\, i+2)$ as a terminal cell created by the surgery.  It is not maximally open: it admits exactly one 3-periodic opening surgery, the inverse of the closing surgery that produced it from $\cD$.

Let $\hat\cD'$ be the dissection obtained from $\hat\cD$ by applying the inverse of the blow-up operation and removing the triangle $(i,\, i+1,\, i+2)$.  It has the same quiddity \eqref{QVblup} as $\tilde\cD'$, and it \textit{is} maximally open, because the only 3-periodic opening surgery admitted by $\hat\cD$ is no longer available.  Hence as before, induction shows that $\hat\cD' = \tilde\cD'$.  But then applying the blow-up operation and reattaching the triangle $(i,\, i+1,\, i+2)$ implies that $\hat\cD = \tilde\cD$, a contradiction, because $\hat\cD$ is not maximally open.
\end{proof}

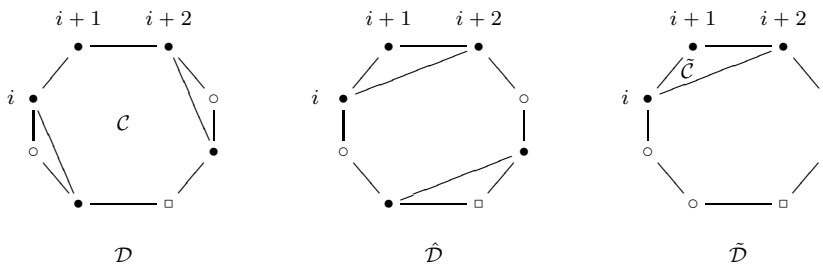
\begin{figure}[htbp]
\footnotesize
$$
\xymatrix @!0 @R=0.35cm @C=0.6cm
{
&& i+1 && i+2 &
\\
&& \bullet \ar@{-}[ldd] \ar@{-}[rr] && \bullet \ar@{-}[rdd]
\\
\\
\mkern36mu i
& \bullet \ar@{-}[dd] \ar@{-}[rdddd] &&&& \circ \ar@{-}[dd]
\\
&&& \cC &&
\\
& \circ \ar@{-}[rdd] &&&& \bullet \ar@{-}[ldd] \ar@{-}[luuuu]
\\
\\
&& \bullet \ar@{-}[rr] && \square
\\
\\
&&& \cD &&
}
\mkern36mu
\xymatrix @!0 @R=0.35cm @C=0.6cm
{
&& i+1 && i+2 &
\\
&& \bullet \ar@{-}[ldd] \ar@{-}[rr] && \bullet \ar@{-}[rdd]
\\
\\
\mkern27mu i
& \bullet \ar@{-}[dd] \ar@{-}[rrruu] &&&& \circ \ar@{-}[dd]
\\
\\
& \circ \ar@{-}[rdd] &&&& \bullet \ar@{-}[ldd]
\\
\\
&& \bullet \ar@{-}[rrruu] \ar@{-}[rr] && \square
\\
\\
&&& \hat\cD &&
}
\mkern27mu
\xymatrix @!0 @R=0.35cm @C=0.6cm
{
&& i+1 && i+2 &
\\
&& \bullet \ar@{-}[ldd] \ar@{-}[rr] && \bullet \ar@{-}[rdd]
\\
&& \!\!\!\tilde\cC
\\
\mkern36mu i
& \bullet \ar@{-}[dd] \ar@{-}[rrruu] &&&& \circ \ar@{-}[dd]
\\
\\
& \circ \ar@{-}[rdd] &&&& \circ \ar@{-}[ldd]
\\
\\
&& \circ \ar@{-}[rr] && \square
\\
\\
&&& \tilde\cD &&
}
$$
\caption*{$\cD$, $\hat\cD$, and $\tilde\cD$ for $j = i+2$: as usual, hollow dots represent sub-dissections, and the base is contained in the square.}
\normalsize
\end{figure}

\section{Counting quiddities}\label{PrThmFMSec}

In this section we will prove Theorem~\ref{DiffQThm}.  Consider the set of all 3-periodic dissections of an $(n+2)$-gon into $m$ cells.  In light of Theorem~\ref{MainTechThm}, the number of maximally open 3-periodic dissections in this set is equal to the number of distinct quiddities of  all dissections in the set.  Thus the function $Q(z, w)$ defined in~\eqref{Q defn} is the BGF of the maximally open 3-periodic dissections.

In order to enumerate the set of maximally open 3-periodic dissections, we must first enumerate a certain subset of these dissections whose BGF satisfies a recursive functional equation.  This BGF will turn out to be the auxiliary function $P(z, w)$ defined in Section~\ref{BGF Sec}.

\begin{defn}
A maximally open 3-periodic dissection is \textit{base-open} if it remains maximally open when the blow-up operation is applied to its base edge.
\end{defn}

The next result is a corollary of Lemma~\ref{3dS count}.  It is similar to Corollary~\ref{MO}: the only difference is that in Corollary~\ref{MO} the base cell is excepted from the $\Z_3$-condition.  This exception is the reason that the generating function $Q(z, w)$ does not satisfy any recursive functional equation we know of.

\begin{cor} \label{MOBO}
A 3-periodic dissection is maximally open and base-open if and only if in every cell, every non-base side of the same $\Z_3$-index as the base side is an edge.
\end{cor}

\begin{ex}
The 3-periodic dissection in Figure~\ref{MO not BO} is maximally open but not base-open, in contrast with the maximally open 3-periodic dissection in Figure~\ref{2 openings}, which is base-open.
\end{ex}

\begin{prop} \label{P BGF}
The function $P(z, w)$ defined in~\eqref{P recurrence} is the BGF of the maximally open base-open 3-periodic dissections: $P_{n, m}$ is the number of such dissections of an $(n+2)$-gon into $m$ cells.
\end{prop}

\begin{proof}
Write $\tilde P_{n, m}$ for the number of maximally open base-open 3-periodic dissections of an $(n+2)$-gon into $m$ cells, and let $\tilde P(z, w)$ be the corresponding BGF, $\sum_{n, m} \tilde P_{n, m} z^n w^m$.  We must prove that $\tilde P(z, w) = P(z, w)$.  It will suffice to prove that $\tilde P(z, w)$ satisfies the same recursive equation~\eqref{P recurrence} satisfied by $P(z, w)$:
\begin{equation} \label{tilde P recurrence}
   \tilde P(z, w) = 1 + wz \tilde P^2 + wz^4 \tilde P^4 + wz^7 \tilde P^6 + wz^{10} \tilde P^8 + \cdots.
\end{equation}

The maximally open base-open 3-periodic dissections have an MGF, analogous to the MGF $D(\vecw)$ of the dissections defined in~\eqref{Dissection MGF defn}.  Following the notation of Section~\ref{MGF sec}, let $P_{\vecm}$ be the number of maximally open base-open 3-periodic $\vecm$-dissections.  The MGF is
\begin{equation*}
   P(\vecw) := \sum_{\vecm \in \N^\omega} P_{\vecm} \vecw^{\vecm}.
\end{equation*}

Consider the effect of replacing $w_r$ by $w z^r$ for all $r$.  As we saw in the proof of Proposition~\ref{Periodic BGF}, this replaces $\vecw^{\vecm}$ by $w^{|m|} z^{\|m\|}$, and so by~\eqref{vecm nm}, it replaces $P(\vecw)$ by $\tilde P(z, w)$.  Therefore \eqref{tilde P recurrence} is a consequence of the following recursive functional equation for $P(\vecw)$:
\begin{equation} \label{PMFE}
   P(\vecw) = 1 + \sum_{d=1}^\infty w_{3d-2} P(\vecw)^{2d}.
\end{equation}

To prove \eqref{PMFE} we follow the proof of Proposition~\ref{Dissection MGF}.  Consider a maximally open base-open 3-periodic $\vecm$-dissection such that the base cell is a $3d$-cell for some positive integer~$d$.  To match earlier notation, set $r = 3d - 2$, so that the base cell is an $(r+2)$-cell.  Label its vertices $v_0, \ldots, v_{r+1}$, where
\begin{equation*}
   0 = v_0 < v_1 < \cdots v_{r-1} < v_r < v_{r+1} = n+1.
\end{equation*}
Observe that the $\Z_3$-index of the side $(v_s, v_{s+1})$ is $s+1$.

For $0 \le s \le r$, consider the sub-dissection induced on the sub-polygon with vertices $v_s, v_s + 1, \ldots, v_{s+1} - 1, v_{s+1}$, which is attached to the base cell along the side $(v_s, v_{s+1})$.  Just as in Proposition~\ref{Dissection MGF}, it is an $\vecm(s)$-dissection for some $\vecm(s)$ such that $\|\vecm(s)\| = v_{s+1} - v_s - 1$.  However, there are now two further conditions, arising from the following obvious statement.

\begin{lem} \label{characterization of MOs and MOBOs}
Fix a 3-periodic dissection $\cD$, and regard it as a collection of 3-periodic sub-dissections attached to the sides of its base cell.
\begin{enumerate}
\item[(i)]
$\cD$ is maximally open if and only if each of the sub-dissections is maximally open and base-open.
\smallbreak \item[(ii)]
If $\cD$ is maximally open, then it is in addition base-open if and only if the sub-dissections attached to the sides of the base cell of $\Z_3$-index $0$ are empty.
\end{enumerate}
\end{lem}

The new conditions are that the $\vecm(s)$-dissection attached to $(v_s, v_{s+1})$ is maximally open and base-open, and empty if $s \equiv 2$ mod~3.  Thus the number of maximally open base-open $\vecm$-dissections such that the base is a $3d$-cell is the sum of all products
\begin{equation*}
   \big(P_{\vecm(0)} P_{\vecm(1)}\big) \big(P_{\vecm(3)} P_{\vecm(4)}\big) \cdots
   \big(P_{\vecm(r-4)} P_{\vecm(r-3)}\big) \big(P_{\vecm(r-1)} P_{\vecm(r)}\big),
\end{equation*}
taken over all choices of $v_1, \ldots, v_r$, and for each such choice, over all choices of
the $\vecm(s)$ such that
\begin{equation*}
   \|\vecm(s)\| = v_{s+1} - v_s - 1 \quad \mbox{\rm and} \quad
   \sum_{i=0}^{d-1} \big(\vecm(3i) + \vecm(3i+1)\big) = \vecm - e_r.
\end{equation*}
This number is identical to the coefficient of $\vecw^{\vecm}$ in $w_{3d-2} P(\vecw)^{2d}$, proving \eqref{PMFE} and hence the proposition.
\end{proof}

Evaluating $P(z, w)$ at $w=1$ gives the following corollary.

\begin{cor}
The function $P(z)$ defined in \eqref{P(z) defn} is the UGF of the maximally open base-open 3-periodic dissections: $P_n$ is the number of such dissections of an $(n+2)$-gon.
\end{cor}

As an aside, let us remark that Theorem~\ref{D coefficients} gives an explicit formula for the coefficients $P_{\vecm}$, via essentially the same trick used in Lemma~\ref{tilde D2}.  To explain, use \eqref{DMFE} to check that the function
\begin{equation*}
   D(w_1, 0, w_4, 0, w_7, 0, w_{10}, 0, \ldots)
\end{equation*}
satisfies the same recursive equation \eqref{PMFE} as $P(\vecw)$.  In other words, $P(\vecw)$ is $D(\vecw)$ with $w_{2d}$ replaced by~$0$ and $w_{2d-1}$ replaced by $w_{3d-2}$ for all $d \ge 1$.  This leads to the following statement.

\begin{prop} \label{P MGF}
The coefficient $P_{\vecm}$ of $\vecw^{\vecm}$ in the MGF $P(\vecw)$ is~$0$ unless $\vecm$ is 3-periodic, in which case it is
\begin{equation*}
   P_{\vecm} = \frac{1}{\frac{2}{3} \|\vecm\| + \frac{1}{3} |\vecm| + 1}
   \binom{\frac{2}{3} \|\vecm\| + \frac{4}{3} |\vecm|}
   {\frac{2}{3} \|\vecm\| + \frac{1}{3} |\vecm|,\ \vecm}.
\end{equation*}
\end{prop}

Finally we have arrived at the proof of our main result:

\medbreak \noindent \textit{Proof of Theorem~\ref{DiffQThm}.}
We follow the proof of Proposition~\ref{P BGF}.  Define the MGF $Q(\vecw)$ of the maximally open 3-periodic dissections by
\begin{equation*}
   Q(\vecw) := \sum_{\vecm \in \N^\omega} Q_{\vecm} \vecw^{\vecm},
\end{equation*}
where $Q_{\vecm}$ is the number of maximally open 3-periodic $\vecm$-dissections.

Check that replacing $w_r$ by $w z^r$ for all $r$ replaces $Q(\vecw)$ by $Q(z, w)$, just as it replaces $P(\vecw)$ by $\tilde P(z, w)$.  Use this to come down to proving the identity
\begin{equation} \label{QMFE}
   Q(\vecw) = 1 + \sum_{d=1}^\infty w_{3d-2} P(\vecw)^{3d-1}.
\end{equation}

Consider maximally open 3-periodic $\vecm$-dissections such that the base cell is a $3d$-cell.  Everything goes as in the proof of~\eqref{PMFE}, except that only Part~(i) of Lemma~\ref{characterization of MOs and MOBOs} applies, so maximally open base-open 3-periodic $\vecm(s)$-sub-dissections can be attached to every non-base side of the base cell, regardless of $\Z_3$-index.

Therefore the number of maximally open 3-periodic $\vecm$-dissections such that the base cell is a $3d$-cell is the sum of all products $\prod_0^r P_{\vecm(s)}$, taken over the same set \eqref{summation set} as in the proof of Proposition~\ref{Dissection MGF}.  This number is identical to the coefficient of $\vecw^{\vecm}$ in $w_{3d-2} P(\vecw)^{3d-1}$, proving \eqref{QMFE} and hence the theorem.
\hfill $\Box$

\begin{rem}
At the beginning of Section~\ref{MaxOpenSec} we claimed that the set of $\ell$-periodic quiddities does not have a well-defined MGF, because the quiddity does not determine $\vecm$.  While this is true, for $\ell = 3$ one could argue that there is a well-defined MGF, namely, $Q(\vecw)$.  It is in fact possible to obtain an explicit formula for the coefficients of $Q(\vecw)$, by proving multivariate analogs of Proposition~\ref{Gen Per KC} and Lemmas~\ref{tilde D2} and~\ref{Pe}.  However, the result does not seem to be either illuminating or elegant.
\end{rem}

\section{Counting blow-ups of the projective plane} \label{ProjPlane}

This section is an application of the formula \eqref{Qn3Eq} for $Q_{n, n-3}$, the number of quiddities of dissections of an $(n+2)$-gon into $n-4$ triangles and a single hexagon.  This is the $k=1$ case of~\eqref{Qnm formula} in Theorem~\ref{TheMainCountThm}.  As we remarked below~\eqref{Qn3Eq}, it was previously obtained in \cite{Gui}.

Consider the rational surfaces obtained from the projective plane
$\bP^2$ by blowing it up at a finite set of points.
They form an important class of toric surfaces which is useful for many purposes;
see for example \cite{McM, ST}.

In general, the result of blow-ups at $n$ distinct points in $\bP^2$
depends on the order in which the blow-ups occur.
However, different orders may give isomorphic surfaces.
The following theorem enumerates the isomorphism classes obtained
in terms of \eqref{Qn3Eq} and the Catalan numbers~\eqref{C_n}.

\begin{thm}
\label{BUThm}
For $n \ge 1$, the number of isomorphism classes of rational surfaces obtained by blow-up at $n$ points is given by the expression
\begin{equation}
\label{BUEq}
   Q_{n+1, n-2} + (n+3) (2C_n-C_{n-1})
   = \frac{n+4}{n} \binom{2n+1}{n-1}.
\end{equation}
\end{thm}

Here $Q_{n+1, n-2}$ should be interpreted as $0$ for $n = 1, 2$.  The simplified formula on the right is valid only for $n \ge 2$.  The sequence begins $4, 15, 49, 168, 594, 2145, \ldots$.

Before giving the proof we briefly recall rational fans and the blow-up operation.  This material is well-known and can be found in Sections~2.4 and~2.5 of the classical book~\cite{Ful}.

\subsection{Rational fans in~$\R^2$}

Every compact nonsingular toric surface is determined by a complete rational fan.
A \textit{regular complete rational fan} in $\R^2$ is an $N$-periodic sequence
of lattice points $v_i\in\Z^2$ satisfying the following two conditions:
\begin{itemize}
\item
every pair $v_i, v_{i+1}$ of consecutive points forms a basis of~$\Z^2$;
\smallbreak \item
distinct cones $(v_i, v_{i+1})$ and $(v_j, v_{j+1})$
intersect only at~$\{0\}$ or along a bounding ray.
\end{itemize}

We shall write such sequences as $(v_i)_{i\in\Z}$,
where $v_{i+N} = v_i$ for all~$i$.
Clearly there is a sequence $(a_i)_{i\in\Z}$ of integers such that
the $v_i$ satisfy the recurrence relation 
\begin{equation}
\label{RecFanEq}
v_{i+1}=a_iv_i-v_{i-1}
\end{equation}
(cf.~\cite{Ful}, p.~43).
Note that the $a_i$ are not necessarily positive.

Applying an element of $\SL(2,\Z)$, we may assume that
$v_1 = e_1$ and $v_2=e_2$.
Hence the fan is determined up to equivalence by the $a_i$.
The following proposition is obvious (see~\cite{Ful}, Exercises~2.17 and~2.18).

\begin{prop}
\label{EasyTor}
\begin{enumerate}
\item[(i)]
The sequence $(a_i)$ in \eqref{RecFanEq} is a solution of \eqref{SLMatEqInt} (not necessarily positive), with $+\Id$ on the right side.  Its total sum \eqref{MatEq Total Sum} is $T = 3N - 12$.

\smallbreak \item[(ii)]
Conversely, any solution of \eqref{SLMatEqInt} of total sum $3N - 12$ determines a complete rational fan with $v_1 = e_1$ and $v_2 = e_2$.
\end{enumerate}
\end{prop}

\begin{exs}
In all cases, we take $v_1 = e_1$ and $v_2 = e_2$.

\begin{enumerate}
\item[(i)]
The first two types of isomorphism classes of rational fans
are represented by a $3$-periodic fan
with $(a_1, a_2, a_3) = (-1,-1,-1)$
and a one-parameter family of $4$-periodic fans with
$(a_1, a_2, a_3, a_4) = (a,0,-a,0)$, where $a\in\Z$.
The corresponding surfaces are the projective plane $\bP^2$
and the Hirzebruch surface $\bF_a$.
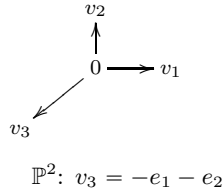
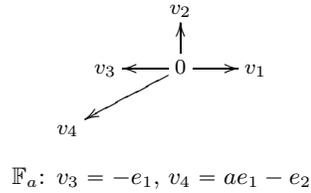
\begin{figure}[!h]
\footnotesize
\begin{subfigure}{.4\textwidth}
$$
\xymatrix @!0 @R=0.8cm @C=1cm
{
& v_2 \ar@{<-}[d]
\\
& 0 \ar@{->}[ld] \ar@{->}[r] & v_1 &
\\
v_3 &
}
$$
\subcaption*{$\bP^2$: $v_3 = -e_1 - e_2$}
\end{subfigure}
\begin{subfigure}{.4\textwidth}
$$
\xymatrix @!0 @R=0.4cm @C=0.5cm
{
&&& v_2 \ar@{<-}[dd]
\\
&&&&&
\\
& v_3 && 0 \ar@{->}[rr] \ar@{->}[ll] && v_1
\\
&&&&&
\\
v_4 \ar@{<-}[rrruu] &&&
}
$$
\subcaption*{$\bF_a$: $v_3 = -e_1$, $v_4 = ae_1 - e_2$}
\end{subfigure}
\addtocounter{figure}{-1}
\normalsize
\end{figure}

\smallbreak \item[(ii)]
The simplest positive solution of \eqref{SLMatEqInt}
with total sum $T = 3N - 12$ is $(a_1,\ldots,a_6)=(1,1,1,1,1,1)$.
It corresponds to the pictured $6$-periodic fan.
By Theorem~\ref{ValRMSthm2}, every positive solution
of \eqref{SLMatEqInt} with $T = 3N - 12$ may be obtained
from this solution by a sequence of blow-ups.
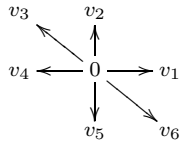
\begin{figure}[!htbp]
\footnotesize
$$
\label{6Fan}
\xymatrix @!0 @R=0.8cm @C=1cm
{v_3 & v_2
\\
v_4 &
0 \ar@{->}[u] \ar@{->}[r] \ar@{->}[rd] \ar@{->}[d] \ar@{->}[l] \ar@{->}[lu] &
v_1 &
\\
& v_5 & v_6
}
$$
\caption*{A fan with $N = 6$: $v_3 = -e_1 + e_2$, $v_4 = -e_1$, $v_5 = -e_2$, $v_6 = e_1 - e_2$}
\normalsize
\end{figure}

\smallbreak \item[(iii)]
An example of a solution with $T = 3N - 6$ is $(a_1, a_2, a_3, a_4) = (1, 2, 1, 2)$.
It gives an anti-periodic quadrilateral of \textit{index}~$\frac{1}{2}$, a ``half-fan''.
By Theorem~\ref{CCthm}, the number of half-fans of $N$ vectors is
the Catalan number~$C_{N-2}$.
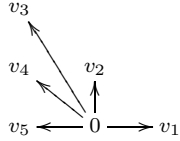
\begin{figure}[!htbp]
\footnotesize
$$
\xymatrix @!0 @R=0.8cm @C=1cm
{
v_3&\\
v_4
&v_2
\\
v_5&
0\ar@{->}[r]\ar@{->}[u]\ar@{->}[luu]\ar@{->}[lu]\ar@{->}[l]&
v_1&
 }
$$
\caption*{A fan with $N = 4$:
$v_3 = -e_1 + 2e_2$, $v_4 = -e_1 + e_2$, $v_5 = -v_1$, $\ldots$
}
\normalsize
\end{figure}

\smallbreak \item[(iv)]
In examples with $T = 3N - 18$ the sequence $(v_i)$ is again anti-periodic, as for $T = 3N - 6$.  Here the index is $\frac{3}{2}$, in the sense that the broken line $(v_i)_{0\leq{}i\leq{}n}$ makes one and a half turns around the origin, as shown.  The number of positive solutions of \eqref{SLMatEqInt} of this type is $Q_{N-2, N-8}$, which is given by~\eqref{Qn6Eq}.
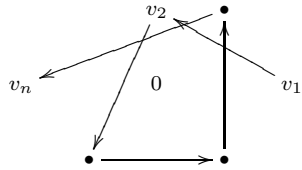
\begin{figure}[!htbp]
\footnotesize
$$
\xymatrix @!0 @R=1cm @C=0.9cm
{
&&v_2\ar@{->}[ldd]&\bullet\ar@{->}[llld]&
\\
v_n&&0&&v_1\ar@{->}[llu]\\
&\bullet&&\bullet\ar@{<-}[ll]\ar@{->}[uu]&
}
$$
\caption*{A fan with $T = 3N - 18$}
\normalsize
\end{figure}

\end{enumerate}
\end{exs}

\subsection{The blow-up operation}

This operation plays a crucial role in the classification of toric surfaces \cite{Ful}.  Recall its combinatorial definition:

\begin{defn}
The \textit{blow-up} of an $N$-periodic fan $(v_i)_{i\in\Z}$ is the $(N+1)$-periodic fan obtained by
inserting the vector $v_k+v_{k+1}$ between $v_k$ and $v_{k+1}$ for some~$k$.
\end{defn}

The corresponding sequence $(a_i)$ changes just as quiddities do under the dissection blow-up operation defined in Section~\ref{BUBA}: $a_k$ and $a_{k+1}$ increase by~$1$, and a~$1$ is inserted between them:
\begin{equation}
\label{BUNewEq}
(\ldots, a_k,\, a_{k+1}, \ldots)
\quad \longmapsto \quad
(\ldots, a_k+1,\, 1,\, a_{k+1}+1, \ldots).
\end{equation}

\medbreak \noindent \textit{Proof of Theorem~\ref{BUThm}.}
The theorem counts the number of distinct fans with $v_1 = e_1$ and $v_2 = e_2$ obtained from the fan of $\bP^2$ by a sequence of blow-ups.  It is not hard to see that the sequence $(a_i)$ obtained by $n \geq 1$ blow-ups of $(-1,-1,-1)$ must be one of the following four mutually exclusive types:

\begin{enumerate}
\item[(a)] 
The $a_i$ are all positive.  Here we must have $n \ge 3$.

\item[(b)] 
There is a $k$ such that $a_k = -1$.  In this case, the neighbors $a_{k-1}$ and $a_{k+1}$ must be non-negative, and the other $a_i$ must all be positive.

\item[(c)] 
There is a $k$ such that $a_k = a_{k+1} = 0$.  In this case, the other $a_i$ must all be positive, and $n \ge 2$.

\item[(d)]
There is a $k$ such that $a_k=0$, and the other $a_i$ are all positive.  Here $n \ge 3$.
\end{enumerate}

We will count the number of sequences of each type separately; combining the results then completes the proof.  Note that after $n$ blow-ups we have an $(n+3)$-periodic sequence.

\begin{itemize}

\item\textit{Type~(a):}
these sequences are the quiddities of 3-periodic dissections of the $(n+3)$-gon into triangles and a single hexagon.  By Theorem~\ref{TheMainCountThm}, there are $Q_{n+1, n-2}$ of them.

\smallbreak \item\textit{Type~(b):}
the following lemma shows that the number of such sequences is $(n+3) C_n$.

\begin{lem} \label{Type B}
The number of solutions of~\eqref{SLMatEqInt} of Type~(b) is $N C_{N-3}$.
\end{lem}

\begin{proof}
Consider the following operation on $N$-periodic sequences of Type~(b):
\begin{equation}
\label{BUDualEq}
(\ldots,\;a_{k-1},\;-1,\;a_{k+1},\;\ldots)
\quad\longmapsto\quad
(\ldots,\;a_{k-1}+1,\;a_{k+1}+1,\;\ldots).
\end{equation}
Check that it produces a positive solution of length $N-1$ and total sum $T = 3N-9$.  By Theorem~\ref{CCthm}, such solutions are quiddities of triangulations of the $(N-1)$-gon, of which there are $C_{N-3}$.  The lemma now follows from the fact that there are $N$ choices for the position $k$ of the $-1$ in the original sequence.
\end{proof}

\smallbreak \item\textit{Type~(c):}
the next lemma shows that the number of such sequences is $(n+3) C_{n-1}$.

\begin{lem}
The number of solutions of \eqref{SLMatEqInt} of Type~(c) is $N C_{N-4}$.
\end{lem}

\begin{proof}
Removing the two consecutive $0$'s gives a positive solution of \eqref{SLMatEqInt} of length $N-2$ and total sum $T = 3N-12$.  The remainder of the proof goes as for Lemma~\ref{Type B}.
\end{proof}

\smallbreak \item\textit{Type~(d):}
our final lemma shows that the number of sequences here is $(n+3) (C_{n}-2C_{n-1})$.

\begin{lem}
The number of solutions of \eqref{SLMatEqInt} of Type~(d) is $N (C_{N-3} - 2C_{N-4})$.
\end{lem}

\begin{proof}
If $(\ldots,\; a_{k-1}, \;0, \; a_{k+1}, \;\ldots)$
is a solution of Type~(d), then
$(\ldots,\; a_{k-1} + a_{k+1},\; \ldots)$ is a positive solution.
By Theorem~\ref{CCthm}, it is the quiddity of a triangulation of the $(N-2)$-gon.

Conversely, given a quiddity $(a_1, \ldots, a_k, \ldots, a_{N-2})$ of a triangulation of the $(N-2)$-gon such that $a_k \ge 2$, any sequence
$$
   (a_1, \ldots, a'_k, 0, a''_k, \ldots, a_{N-2})
$$
with $a'_k + a''_k = a_k$ and $a'_k$, $a''_k$ positive corresponds to splitting the triangulation into two triangulations along a chord.  Therefore the number of solutions with $0$ at a fixed position is $\sum_{i=1}^{N-5} C_i C_{N-4-i}$.  By the quadratic recurrence equation for the Catalan numbers, this is $C_{N-3} - 2C_{N-4}$.  The lemma now follows as before.
\end{proof}

\end{itemize}

\begin{rem}
The operation \eqref{BUDualEq} used for Type~(b) has an inverse:
$$
   (\ldots,\;a_{k},\;a_{k+1},\;\ldots)
   \quad\longmapsto\quad
   (\ldots,\;a_{k}-1,\;-1,\;a_{k+1}-1,\;\ldots).
$$
This is a ``negative version'' of the blow-up operation~\eqref{BUNewEq}.
In terms of the sequence $(v_i)$, it inserts the vector $v_{k+1}-v_k$ between $v_k$ and $v_{k+1}$ and reverses the sign of the subsequent vectors, giving
$$
   (\ldots,\;v_{k},\;v_{k+1}-v_k,\;-v_{k+1},\;-v_{k+2},\;\ldots).
$$
\end{rem}

\begin{exs}
We conclude with descriptions of the sequences counted by Theorem~\ref{BUThm} for $n \le 4$.

\begin{itemize}

\item\textit{$n=1$, 4-periodic sequences:}
blowing up $\bP^2$ at one point gives $4$ sequences $(a_i)$, all of Type~(b):
the cyclic permutations of $(-1, 0, 1, 0)$.  The corresponding fans are
\begin{align*}
   & \big\{e_1,\ e_2,\ -e_1,\ -e_1 - e_2 \big\}, \qquad
   & \big\{e_1,\ e_2,\ -e_1 - e_2,\ -e_2 \big\}, \\[6pt]
   & \big\{e_1,\ e_2,\ -e_1,\ e_1 - e_2 \big\}, \qquad
   & \big\{e_1,\ e_2,\ -e_1 + e_2,\ -e_2 \big\}.
\end{align*}

\medbreak \item\textit{$n=2$, 5-periodic sequences:}
blowing up at two points gives $15$ sequences, $10$ of Type~(b) and $5$ of Type~(c), the cyclic permutations of the following:
\begin{itemize}
\item{Type~(b):}
$(-1, 0, 2, 1, 1)$, $(-1, 1, 1, 2, 0)$;
\smallbreak \item{Type~(c):}
$(0, 0, 1, 1, 1)$.
\end{itemize}

\medbreak \item\textit{$n=3$, 6-periodic sequences:}
blowing up at three points gives $49$ sequences: $1$ of Type~(a), $30$ of Type~(b), $12$ of Type~(c), and $6$ of Type~(d), the cyclic permutations of the following:
\begin{itemize}
\item{Type~(a):}
$(1, 1, 1, 1, 1, 1)$;
\smallbreak \item{Type~(b):}
$(-1, 0, 2, 2, 1, 2)$, $(-1, 2, 1, 2, 2, 0)$, \\ \phantom{Type~(b)\ }
$(-1, 0, 3, 1, 2, 1)$, $(-1, 1, 2, 1, 3, 0)$, $(-1, 1, 1, 3, 1, 1)$;
\smallbreak \item{Type~(c):}
$(0, 0, 1, 2, 1, 2)$, $(0, 0, 2, 1, 2, 1)$;
\smallbreak \item{Type~(d):}
$(0, 1, 1, 2, 1, 1)$.
\end{itemize}

\medbreak \item\textit{$n=4$, 7-periodic sequences:}
blowing up at four points gives $168$ sequences: $7$ of Type~(a), $7 \cdot 14$ of Type~(b), $7 \cdot 5$ of Type~(c), and $7 \cdot 4$ of Type~(d).  We will not list representatives of each cyclic permutation class, but let us give the Type~(d) cases:
$$
   (0, 1, 1, 2, 2, 1, 2), \quad
   (0, 2, 1, 2, 2, 1, 1), \quad
   (0, 1, 2, 1, 3, 1, 1), \quad
   (0, 1, 1, 3, 1, 2, 1).
$$
There are four of them because $4 = C_4 - 2C_3 = C_1 C_2 + C_2 C_1$.  They correspond, respectively, to the divisions of triangulations of the pentagon indicated by the figure.

\begin{figure}[htbp]
\begin{subfigure}{.2\textwidth}
$$
\xymatrix @!0 @R=0.50cm @C=0.5cm
{
&&201\ar@{-}[rrd]\ar@{-}[lld]\ar@{-}[lddd]\ar@{=}[rddd]&
\\
1\ar@{-}[rdd]&&&& 1\ar@{-}[ldd]\\
\\
&2\ar@{-}[rr]&& 2
}
$$
\end{subfigure}
\begin{subfigure}{.2\textwidth}
$$
\xymatrix @!0 @R=0.50cm @C=0.5cm
{
&&102\ar@{-}[rrd]\ar@{-}[lld]\ar@{=}[lddd]\ar@{-}[rddd]&
\\
1\ar@{-}[rdd]&&&& 1\ar@{-}[ldd]\\
\\
&2\ar@{-}[rr]&& 2
}
$$
\end{subfigure}
\begin{subfigure}{.2\textwidth}
$$
\xymatrix @!0 @R=0.50cm @C=0.5cm
{
&&101\ar@{-}[rrd]\ar@{-}[lld]\ar@{=}[lddd]&
\\
1\ar@{-}[rdd]&&&& 2\ar@{-}[ldd]\ar@{-}[llldd]\\
\\
&3\ar@{-}[rr]&& 1
}
$$
\end{subfigure}
\begin{subfigure}{.2\textwidth}
$$
\xymatrix @!0 @R=0.50cm @C=0.5cm
{
&&101\ar@{-}[rrd]\ar@{-}[lld]\ar@{=}[rddd]&
\\
2\ar@{-}[rdd]\ar@{-}[rrrdd]&&&& 1\ar@{-}[ldd]\\
\\
&1\ar@{-}[rr]&& 3
}
$$
\end{subfigure}
\end{figure}

\end{itemize}
\end{exs}

\end{document}